\documentclass[a4paper,11pt]{amsart}

\usepackage{amsfonts, latexsym, amsmath, amssymb, amsthm, amscd}
\usepackage[all]{xy}
\usepackage[english]{babel}
\usepackage[utf8]{inputenc}
\usepackage{currvita}
\usepackage{graphicx}
\usepackage{booktabs}
\usepackage{array, tabularx}
\usepackage{textcomp}
\usepackage{tikz, tikz-cd}
\usepackage{multirow}
\usepackage{comment}
\usepackage{url}
\usepackage{tensor}
\usepackage{MnSymbol}
\usepackage{color}
\usepackage[hypertexnames=false,
backref=page,
    pdftex,
    pdfpagemode=UseNone,
    breaklinks=true,
    extension=pdf,
    colorlinks=true,
    linkcolor=blue,
    citecolor=blue,
    urlcolor=blue,
]{hyperref}

\usepackage{enumitem}
\usepackage[top=1.5in, bottom=.9in, left=0.9in, right=0.9in]{geometry}

\theoremstyle{definition}
\newtheorem*{notation*}{Notation}
\theoremstyle{plain}
\newtheorem{theorem}{Theorem}[section]
\newtheorem*{theorem*}{Theorem}
\newtheorem{definition}[theorem]{Definition}
\newtheorem{lemma}[theorem]{Lemma}
\newtheorem{prop}[theorem]{Proposition}
\newtheorem*{prop*}{Proposition}
\newtheorem{cor}[theorem]{Corollary}
\newtheorem*{cor*}{Corollary}
\theoremstyle{definition}
\newtheorem{rem}[theorem]{Remark}
\newtheorem{ex}[theorem]{Example}
\newtheorem*{mt*}{Main Theorem}
\newtheorem*{question}{Question}
\newtheorem*{acknowledgements}{Acknowledgements}

\DeclareMathOperator{\Span}{Span}

\DeclareMathOperator{\Ker}{Ker}

\DeclareMathOperator{\im}{Im}
\DeclareMathSymbol{\Finv} {\mathord}{AMSb}{"60}

\newcommand\restrict[1]{\raisebox{-.5ex}{$|$}_{#1}}

\newcommand{\R}{\mathbb{R}}
\newcommand{\C}{\mathbb{C}}
\newcommand{\Z}{\mathbb{Z}}
\newcommand{\q}{\mathbb{Q}}
\newcommand{\del}{\partial}
\newcommand{\delbar}{\overline{\partial}}

\numberwithin{equation}{section}

\DeclareFontFamily{U}{MnSymbolC}{}
\DeclareSymbolFont{MnSyC}{U}{MnSymbolC}{m}{n}
\DeclareFontShape{U}{MnSymbolC}{m}{n}{
    <-6>  MnSymbolC5
   <6-7>  MnSymbolC6
   <7-8>  MnSymbolC7
   <8-9>  MnSymbolC8
   <9-10> MnSymbolC9
  <10-12> MnSymbolC10
  <12->   MnSymbolC12}{}
\DeclareMathSymbol{\intprod}{\mathbin}{MnSyC}{'270}

\allowdisplaybreaks

\author{Tommaso Sferruzza}
\address[Tommaso Sferruzza]{
Dipartimento di Scienze Matematiche, Fisiche e Informatiche\\
Unità di Mate\-matica e Informatica\\
Università degli studi di Parma}
\email{tommaso.sferruzza@unipr.it}
\author{Adriano Tomassini}
\address[Adriano Tomassini]
{Dipartimento di Scienze Matematiche, Fisiche e Informatiche\\
Unità di Mate\-matica e Informatica\\
Università degli studi di Parma}
\email{adriano.tomassini@unipr.it}

\title[Bott-Chern formality and Massey products\dots]{Bott-Chern formality and Massey products on strong K\"ahler with torsion and K\"ahler solvmanifolds}

\thanks{The authors have been supported by GNSAGA of INdAM. The first author has been supported by Ludwig-Maximilians-Univers\"at M\"unchen during his research visit. The second author has been supported by the Project PRIN: Real and Complex Manifolds: Geometry and Holomorphic Dynamics 2022AP8HZ9}

\date{\today}

\begin{document}
\begin{abstract}
We study the interplay between geometrically-Bott-Chern-formal metrics and SKT metrics. We prove that a $6$-dimensional nilmanifold endowed with a invariant complex structure admits an SKT metric if and only if it is geometrically-Bott-Chern-formal. We also provide some partial results in higher dimensions for nilmanifolds endowed with a class of suitable complex structures. Furthermore, we prove that any K\"ahler solvmanifold is geometrically formal. Finally, we explicitly construct lattices for a complex solvable Lie group in the list of Nakamura \cite{Nak} on which we provide a non vanishing quadruple $ABC$-Massey product.
\end{abstract}

\maketitle

\section{Introduction}

Let $(M,J)$ be a compact complex manifold of real dimension $2n$. Dolbeault, Bott-Chern, and Aeppli cohomologies play a relevant role in the study of the holomorphic invariants and provide obstructions for the existence of further structures. For example, the existence of a K\"ahler metric on $(M,J)$ implies that the $(p,p)$-Hodge numbers are positive, the $\del\delbar$-Lemma holds, the de Rham complex is a formal DGA \cite{Sul75,Sul77,DGMS} and $M$ satisfies the Hard Lefschetz condition.
When $(M,J)$ admits a K\"ahler metric, all these cohomology groups are isomorphic. This fact is no longer true for compact non K\"ahler manifolds; in such a case, Bott-Chern cohomology may yield further information on the complex geometry of the manifold.

Compact quotiens of simply connected nilpotent Lie groups by a lattice, namely \emph{nilmanifolds}, endowed with an invariant complex structure, are one of the main sources of concrete examples on which explicit cohomological computations can be carried out. Due to Benson and Gordon \cite{BG}, such manifolds never admit a K\"aher structure unless they are tori. Nevertheless, nilmanifolds can admit special Hermitian metrics, e.g., \emph{strong K\"ahler with torsion metrics}, shortly \emph{SKT metrics} (see \cite{AG,CG,FT09S, GHR,Gua,Hit} for general results on SKT geometry and generalized K\"ahler geometry and \cite{FPS,RosTom,FGV,FT09,FV} for existence results on nilmanifolds).

Bott-Chern cohomology provides an obstruction to the existence of curves of SKT metrics starting from a SKT metric (see \cite{PS}). Another relation between Bott-Chern cohomology and metric properties is given by the notion of \emph{Bott-Chern formality}. More precisely, following Kotschick \cite{Kot}, in \cite{AngTom} the notion of \emph{geometrically-Bott-Chern-formal metrics}, shortly \emph{geometrically-$BC$-formal metrics}, is defined. In the same work, as an obstruction to the existence of geometrically-$BC$-formal metrics, the authors introduced the \emph{triple Aeppli-Bott-Chern-Massey products}, shortly \emph{triple ABC-Massey products}; such products are defined by three cohomology Bott-Chern classes satisfying suitable conditions which produce a coholomogy class living in a quotient of Aeppli cohomology modulo an ideal of indeterminacy.
Recently, Milivojevic and Stelzig \cite{MS} introduced the \emph{n-fold Aeppli-Bott-Chern-Massey products}, generalizing the triple $ABC$-Massey products in \cite{AngTom}.

In the first part of the present paper, we investigate possible further interplays between SKT metrics and geometrically-$BC$-formal metrics on nilmanifolds endowed with an invariant complex structure. We prove the following.
\begin{theorem*}[see Theorem \ref{thm:SKT-geomBC-dim3}]
Let $(\Gamma\backslash G:=M,J)$ be a $6$-dimensional nilmanifold endowed with an invariant complex structure $J$. Then, $(M,J)$ is SKT if, and only if, it is geometrically-$BC$-formal. In particular, every Hermitian invariant metric is SKT if, and only if, it is geometrically-$BC$-formal.
\end{theorem*}
For $n>3$, there exists SKT nilmanifolds with no geometrically-$BC$-formal metrics (see Proposition \ref{prop:no_sp_type_SKT}). Consequently, for $n>3$, it is natural to look for special classes of complex structures, in order to extend Theorem \ref{thm:SKT-geomBC-dim3}. In Proposition \ref{prop:geomBC_SKT} we prove that, under suitable assumptions on the complex structure, see \eqref{eq:struct_eq}, if the nilmanifold has an invariant geometrically-$BC$-formal metric, then it also admits an SKT metric. This is a partial converse of \cite[Theorem 7.4]{SfeTom2}.

The second part of the paper is devoted to the investigation of cohomological properties of solvmanifolds endowed with an invariant complex structure. By a \emph{solvmanifold}, we mean a compact quotient of a simply connected solvable Lie group by a lattice. In \cite{Hase}, Hasegawa fully characterized solvmanifolds admitting a K\"ahler structure. Namely, he proved that solvmanifolds admit a K\"ahler structure if and only if they are finite quotients of a tori with a structure of a holomorphic bundle over a torus with fiber a torus. Following the explicit description by Hasegawa, we prove
\begin{theorem*}[see Theorem \ref{thm:formality_KS} and Corollary \ref{cor:ABC_prod_KS}]
Let $X$ be any K\"ahler solvmanifold. Then, $X$ is geometrically-$BC$-formal and, consequently, every $ABC$-Massey on $X$ vanishes.
\end{theorem*}
Therefore, the vanishing of $ABC$-Massey products provides an obstruction to the K\"ahlerianity of solvmanifolds. It is useful to observe that in \cite{SfeTom1}, the authors construct an explicit example of a compact complex manifold satisfying the $\del\delbar$-lemma and admitting a triple non vanishing $ABC$-Massey product. As already asked by Milivojevic and Stelzig in \cite{MS}, it is natural to rise the following
\begin{question}
Are $ABC$-Massey products an obstruction to the existence of a K\"ahler metric on a given compact complex manifold?
\end{question}
Concerning the explicit construction of higher $ABC$-Massey products, we provide a non vanishing quadruple $ABC$-Massey product on a complex solvmanifold of complex dimension $4$.

The work is organized as follows. In Section \ref{sec:2} we start by fixing the notation and recalling the basic facts on $ABC$-Massey products.  Section \ref{sec:3} is devoted to the proof of Theorem \ref{thm:SKT-geomBC-dim3}. The argument relies on the characterization, due to Salamon \cite{Sal}, of global frames for which $J$ can be written in a convenient way. In Section \ref{sec:4}, complex structures of \emph{special type} on nilmanifolds are introduced. For such complex structures, Lemma \ref{lemma:BCcohom} gives a concrete description of their Bott-Chern cohomology and Theorem \ref{thm:obstruction} gives a sufficient condition for the existence of a non vanishing $ABC$-Massey product on nilmanifolds endowed with complex structures of special type. Theorem \ref{thm:obstruction} is then applied explicitly in Example \ref{example:1} and Example \ref{example:2}.
Section \ref{sec:5} contains the proof of Theorem \ref{thm:formality_KS}. Such a result follows from the full description of the cohomology of K\"ahler solvmanifolds proved in Theorem \ref{thm:cohom_KS}. Finally, in Section \ref{sec:6} we construct a non vanishing quadruple $ABC$-Massey product of the quotient of a complex Lie group of complex dimension $4$ in the list of Nakamura \cite[IV 6., p. 108]{Nak}. We explicitly describe two families of lattices (Lemma \ref{lemma:lattices_quad}) giving rise to two complex parallelizable manifolds with different Dolbeault and Bott-Chern cohomologies (Table \ref{table:1} and Table \ref{table:2}) which have a structure of holomorphic fiber bundle over a complex torus with non K\"ahler typical fiber, i.e., the \emph{Iwasawa} manifold.

\begin{acknowledgements}
The first author would like to thank Jonas Stelzig for his kind hospitality at Ludwig-Maximilians-Universit\"at  M\"unchen and the authors would like to thank him for many useful discussions and remarks. The authors would also like to thank Beatrice Brienza for pointing out reference \cite{Yam}.
\end{acknowledgements}
\section{Notation and preliminary}\label{sec:2}
We denote by $M=\Gamma\backslash G$ the quotient of a simply connected solvable Lie group $G$ with Lie algebra $\mathfrak{g}$ by a lattice $\Gamma$ and by $J$ an invariant complex structure, i.e., a complex structure on $\mathfrak{g}$. Let $\mathfrak{g}^*$ be the dual of $\mathfrak{g}$ and let $J$ define a complex structure as usual  on $\mathfrak{g}^*$ by $J\alpha(X)=\alpha(JX)$, for every $\alpha\in \mathfrak{g}^*$, $X\in \mathfrak{g}$. Then, the space $\mathfrak{g}_{\C}^*:=\mathfrak{g}^*\otimes \C$ decomposes as
\[
\mathfrak{g}_{\C}^*=(\mathfrak{g}_{\C}^*)^{1,0}\oplus (\mathfrak{g}_{\C}^*)^{0,1}
\]
where $(\mathfrak{g}_{\C}^*)^{1,0}=\{\alpha-iJ\alpha:\alpha\in \mathfrak{g}^*\}$ and $(\mathfrak{g}_{\C}^*)^{0,1}=\{\alpha+iJ\alpha:\alpha\in \mathfrak{g}^*\}$ are, respectively, the $\pm i$-eigenspaces of the $\C$-linear extension of $J$ on $\mathfrak{g}_{\C}^*$.

For a compact complex manifold $(M,J)$ we denote by $\mathcal{A}^{\bullet}_{\C}(M)$ and $\mathcal{A}^{\bullet,\bullet}(M)$ respectively, the space of complex-valued differential forms on $(M,J)$, and the space of $(p,q)$-forms on $(M,J)$. Set, as usual, 
$$\del =\Pi^{1,0}\circ d:\mathcal{A}^{p,q}(M)\to 
\mathcal{A}^{p+1,q}(M)\,\qquad \delbar =\Pi^{0,1}\circ d:\mathcal{A}^{p,q}(M)\to 
\mathcal{A}^{p,q+1}(M).$$  Accordingly, the exterior differential decomposes as $d=\del +\delbar$. Then, 
\[
H_{BC}^{\bullet,\bullet}(M):=\frac{\ker \del \cap \ker\delbar}{\im \del\delbar}\cap \mathcal{A}^{\bullet,\bullet}(M)
\]
is the \emph{Bott-Chern cohomology} of $(M,J)$ and we denote by
\[
H_{A}^{\bullet,\bullet}(M):=\frac{\ker \del\delbar}{\im \del+\im \delbar}\cap \mathcal{A}^{\bullet,\bullet}(M)
\]
the \emph{Aeppli cohomology} of $(M,J)$.

According to \cite{Schw}, once fixed an Hermitian metric $g$ on $(M,J)$, setting the \emph{Bott-Chern Laplacian} and the \emph{Aeppli Laplacian} respectively as
\begin{gather*}
\Delta_{BC}:= \del\delbar\delbar^*\del^* +\delbar^*\del^*\del\delbar + \delbar^*\del \del^*\delbar + \del^*\delbar\delbar^*\del +\delbar^*\delbar +\del^*\del,\\
\Delta_{A}:= \del\del^*+\delbar\delbar^*+ \delbar^*\del^*\del\delbar +\del\delbar\delbar^*\del^*+\del\delbar^*\delbar\del^*+\delbar\del^*\del\delbar^*,
\end{gather*}
it turns out that $\Delta_{BC}$ and $\Delta_{A}$ are fourth order self-adjonint elliptic operators and, consequently, the following isomorphisms of vector spaces hold
\[
\ker\Delta_{BC}\cong H_{BC}(M), \qquad  \ker\Delta_{A}\cong H_{A}(M).
\]
In particular, a $(p,q)$-form is \emph{Bott-Chern harmonic} (respectively, \emph{Aeppli harmonic}) if, fixed a Hermitian metric $g$ on $(M,J)$, it holds 
\[
d\alpha=0, \qquad \del\delbar\ast_g\alpha=0,
\]
respectively,
\[
\del\delbar\alpha=0, \qquad \del\ast_g\alpha=0, \qquad \delbar\ast_g\alpha=0.
\]
The Bott-Chern cohomology of a compact complex manifold has a structure of algebra induced by the $\cup$ product of cohomology classes and the Aeppli cohomology has a structure of $H_{BC}$-module induced by the $\cup$ product, whereas this is no longer true for their harmonic representatives. Therefore, the authors in \cite{AngTom} give the following definition, on the lines of Kotschick in \cite{Kot} and the second author and Torelli in \cite{TomTor}.
\begin{definition}
A Hermitian metric $g$ on $(M,J)$ is said \emph{geometrically-Bott-Chern-formal} (shortly, \emph{geometrically-$BC$-formal}) if the space of \emph{Bott-Chern harmonic forms} 
$\mathcal{H}_{\Delta_{BC}}^{\bullet,\bullet}(M):=\ker\Delta_{BC}$ has a structure of algebra induced by the $\wedge$ product.
\end{definition} 

We recall the construction of \emph{triple $ABC$-Massey products} and \emph{quadruple $ABC$-Massey products}. We refer to \cite{AngTom} and \cite{MS} for further details on, respectively, the triple $ABC$-Massey products and the quadruple $ABC$-Massey products. 

Let $[\alpha]\in H_{BC}^{p,q}(M)$, $[\beta]\in H_{BC}^{r,s}(M)$, and $[\gamma]\in H_{BC}^{u,v}(M)$ such that
\[
\alpha\wedge\beta=(-1)^{p+q}\del\delbar f_{\alpha\beta}, \qquad \beta\wedge\gamma=(-1)^{r+s}\del\delbar f_{\beta\gamma}.
\] 
Then, the \emph{triple $ABC$-Massey product} $\langle[\alpha],[\beta],[\gamma]\rangle_{ABC}$ is represented by
\[
\langle[\alpha],[\beta],[\gamma]\rangle_{ABC}=[(-1)^{p+q}\alpha\wedge f_{\beta\gamma}-(-1)^{r+s}f_{\alpha\beta}\wedge\gamma]\in\frac{H_A^{p+r+u-1,q+s+v-1}(M)}{[\alpha]\cup H_A^{r+u-1,s+v-1}(M)+[\gamma]\cup H_A^{p+r-1,q+s-1}(M)}.
\]
\begin{lemma}\label{lemma:corrinvABC}
Let $(M=\Gamma\backslash G,J)$ be the compact quotient of simply connected Lie group $G$ by a discrete co-compact subgroup $\Gamma$ and let $J$ be an invariant complex structure on $M$, i.e., induced by an invariant structure $J$ on the Lie algebra $\mathfrak{g}$ of $G$.  Then, the map
\[
i\colon\left\{\text{triple ABC-Massey products on} \,\,(\mathfrak{g},J)\right\}\hookrightarrow \left\{\text{triple ABC-Massey products on} \,\,(M,J)\right\} 
\]
is injective.
\end{lemma}
\begin{proof}
Let $[\alpha],[\beta],[\gamma]\in H_{BC}(\mathfrak{g})$ so that $\mathcal{P}:=\langle[\alpha],[\beta],[\gamma]\rangle_{ABC}$ is a non vanishing triple $ABC$-Massey product on $\mathfrak{g}$ with $f_{\alpha\beta}$ and $f_{\beta\gamma}$ its associated invariant $\del\delbar$-primitives. Being $\mathcal{P}$ non vanishing, for the representative of $\mathcal{P}$ it holds that 
\begin{equation}\label{eq:lemma_ABC_triple}
(-1)^{|\alpha|}\alpha\wedge f_{\beta\gamma}-(-1)^{|\beta|}f_{\alpha\beta}\wedge \gamma\notin [\alpha]\cup H_A(\mathfrak{g})+[\gamma]\cup H_A(\mathfrak{g}).
\end{equation}
By contradiction let us assume that $i(\mathcal{P})$ is trivial as a triple $ABC$-Massey product on $(M,J)$. We can choose again $f_{\alpha\beta}$ and $f_{\beta\gamma}$ as $\del\delbar$-primitives, as any other choice of (not necessarily invariant) primitives, by definition of $ABC$-Massey product, would yield terms of the representative that still belong to the indetermincy ideal. Since $i(\mathcal{P})$ is trivial on $(M,J)$, there exist  $\xi,\eta, R, S\in\mathcal{A}(M)$ such that
\[
0\neq (-1)^{|\alpha|}\alpha\wedge f_{\beta\gamma}-(-1)^{|\beta|}f_{\alpha\beta}\wedge \gamma= \alpha\wedge \xi +\gamma\wedge\eta +\del R+\delbar S.
\]
Note that at least one form in $\{\xi,\eta,R,S\}$ is not left-invariant. 
However, by applying a symmetrization process (see, e.g., \cite{B,FG,U}) on both sides of the above equation, we obtain
\[
0\neq (-1)^{|\alpha|}\alpha\wedge f_{\beta\gamma}-(-1)^{|\beta|}f_{\alpha\beta}\wedge \gamma= \alpha\wedge \tilde{\xi} +\gamma\wedge\tilde{\eta} +\del\tilde{R}+\delbar\tilde{S}.
\]
The right-hand side of the equations, however, has to be non vanishing and it belongs to $[\alpha]\cup H_A(\mathfrak{g})+[\gamma]\cup H_A(\mathfrak{g})$, contraddicting  \eqref{eq:lemma_ABC_triple}, which concludes the proof.\\
\end{proof}

We now recall the construction of quadruple $ABC$-Massey products. The {\em Schweitzer complex} $S_{p,q}^k$ is defined for every pair $(p,q)$ as
\begin{gather*}
\dots \xrightarrow{\del\delbar} S_{p,q}^{-2}=\mathcal{A}^{p-2,q-2}(M)\xrightarrow {\del+\delbar} \quad S_{p,q}^{-1} =\mathcal{A}^{p-1,q-2}(M)\oplus \mathcal{A}^{p-2,q-1}(M)\xrightarrow{\text{pr}\circ d} S_{p,q}^{0}=
\mathcal{A}^{p-1,q-1}\\
\xrightarrow{\del\delbar} S_{p,q}^{1}=\mathcal{A}^{p,q}(M)\xrightarrow{\del +\delbar}\dots
\end{gather*}
so that $$H_{S_{p,q}}^1=H_{BC}^{p,q}(M), \qquad H_{S_{p,q}}^{0}=H_A^{p-1,q-1}(M).$$
Let $[\alpha]\in H_{BC}(M)$, $[\beta]\in H_{BC}(M)$, $[\gamma]\in H_{BC}(M)$, and $[\delta]\in H_{BC}(M)$   such that $\alpha\wedge\beta\wedge\gamma\wedge\delta\in\mathcal{A}^{p,q}(M)$. Let $x,y,z,\eta,\eta',\xi,\xi'\in\mathcal{A}^{\bullet,\bullet}(M)$ such that 
\[
\del\delbar x=\alpha\wedge\beta,\quad \del\delbar y=\beta\wedge\gamma,\quad \del\delbar z=\gamma\wedge\delta
\]
and
\[
x\gamma-\alpha y=\del\eta +\delbar \eta', \quad y\delta-\beta z=\del\xi +\delbar \xi'.
\]
Then the quadruple $ABC$-Massey product $\langle[\alpha],[\beta],[\gamma],[\delta]\rangle_{ABC}$ is represented by
\begin{equation}\label{eq:ABC_quad}
[(-1)^{|\alpha|}\alpha\wedge(\xi+\xi')-(\del x)\wedge z+(-1)^{|x|+1}x\wedge\delbar z+(\eta+\eta')\wedge\delta]
\end{equation}
in
\begin{equation}\label{eq:H_S_-1_pq_quad}
\resizebox{.18\textwidth}{!}{$ H_{S_{p,q}}^{-1}(\mathcal{A}^{\bullet,\bullet}(M))$}\resizebox{.80\textwidth}{!}{$\displaystyle =\frac{\ker(\text{pr}\circ d\colon \mathcal{A}^{p-2,q-1}(M)\oplus \mathcal{A}^{p-1,q-2}(M)\rightarrow \mathcal{A}^{p-1,q-1}(M))}{\im(\text{pr}\circ d\colon \mathcal{A}^{p-3,q-1}(M)\oplus \mathcal{A}^{p-2,q-2}(M)\oplus\mathcal{A}^{p-1,q-3}(M)\rightarrow\mathcal{A}^{p-2,q-1}(M)\oplus \mathcal{A}^{p-1,q-2}(M))}.$}
\end{equation}

Let $(M=\Gamma\backslash G,J)$ be a $2k$-dimensional nilmanifold endowed with an invariant almost complex structure $J$ induced by a frame $\{\eta^1,\dots, \eta^k\}$ of $(\mathfrak{g}_{\C}^*)^{1,0}$ such that
\begin{equation}\label{eq:struct_eq}
\begin{cases}
d\eta^j=0, \quad j\in \{1,\dots,k-1\}\\
d\eta^k\in\text{Span}_{\C}\langle \eta^{ij}, \eta^{i\overline{j}}\rangle_{i,j=1,\dots, k-1}
\end{cases}.
\end{equation}
We will call $J$ a \emph{special type complex structure}.

Finally, we recall that a Hermitian metric $g$ on a complex manifold $(M,J)$ is said to be \emph{strong K\"ahler with torsion}, shortly SKT, if its fundamental form $\omega$ satisfies
\[
\del\delbar\omega=0.
\]
\section{SKT metrics and geometrically-$BC$-formal metrics in complex dimension 3}\label{sec:3}
In this section, we prove that the property of admitting a geometrically-$BC$-formal metric is equivalent to the property of admitting a SKT metric on any six-dimensional nilmanifold endowed with an invariant complex structure.

\begin{theorem}\label{thm:SKT-geomBC-dim3}
Let $(\Gamma\backslash G:=M,J)$ be a $6$-dimensional nilmanifold endowed with an invariant complex structure $J$. Then, $(M,J)$ is SKT if, and only if, it is geometrically-$BC$-formal. In particular, every Hermitian invariant metric is SKT if, and only if, it is geometrically-$BC$-formal.
\end{theorem}
\begin{proof}
Let us first assume that $(M,J)$ admits a SKT metric, which we can assume to be invariant by the symmetrization process. By \cite[Theorem 1.2]{FPS}, then $(M,J)$ admits a basis $\{\alpha^1,\alpha^2,\alpha^3\}$ of invariant $(1,0)$-forms such that
\[
\begin{cases}
d\alpha^1=0, \qquad d\alpha^2=0,\\
d\alpha^3=A\alpha^{12}+B\alpha^{1\overline{1}}+C\alpha^{1\overline{2}}+D\alpha^{2\overline{1}}+E\alpha^{2\overline{2}},
\end{cases}
\]
with $A,B,C,D,E\in\C$ satisfying $|A|^2+|C|^2+|D|^2=2\mathfrak{Re}(\overline{B}D)$. In particular, every invariant metric on $(M,J)$ is SKT. But then, in \cite[Theorem 7.3]{SfeTom2} it was shown that every invariant Hermitian metric on $(M,J)$ is also geometrically-$BC$-formal.

Viceversa, let us assume that the nilmanifold $(M=\Gamma\backslash G,J)$ is not SKT. Then, we will prove that the Lie algebra $(\mathfrak{g},J)$ of $G$ admits a non vanishing $ABC$-Massey product, which corresponds to a non vanishing $ABC$-Massey product on $(M,J)$ by Lemma \ref{lemma:corrinvABC}, thus proving that $(M,J)$ is not geometrically-$BC$-formal.

Since the nilmanifold $M=\Gamma\backslash G$ admits an invariant complex structure $J$, by \cite[Theorem 2.5]{Sal}, we can assume that the dual $\mathfrak{g}^*$ of the Lie algebra $\mathfrak{g}$ of $G$ admits a real basis $\{e^1,\dots,e^6\}$ with structure equations
\begin{equation}\label{eq:struct_nil3_real}
\begin{cases}
d e^1=0,\\
d\eta^2=0,\\
d e^3= A_1 e^{12},\\
de^4= B_1 e^{12} + B_2 e^{13} + B_3 e^{23},\\
de^5=C_1 e^{12} +C_2 e^{13} + C_3 e^{14} + C_4 e^{23} + C_5 e^{24} + C_6 e^{34},\\
de^6= D_1 e^{12} + D_2 e^{13} + D_3 e^{14} + D_4 e^{15} + D_5 e^{23} + D_6 e^{24} + D_7 e^{25} + D_8 e^{34},
\end{cases}
\end{equation}
for $A_1,B_i,C_j,D_k\in\q$. We point out that the Jacoby identity, i.e., $d^2=0$, implies that the coefficients must satisfy
\begin{align}\label{eq:Jacobi_nil3}
\begin{cases}
-B_3D_3 - C_4D_4 + B_2D_6 + C_2D_7 - B_1D_8=0,\\
-B_3C_3 + B_2C_5 - B_1C_6=0, \\ 
-C_5D_4 + C_3D_7 + A_1D_8=0, \\  
A_1C_6=0\\
C_6D_4=0, \\
C_6D_7=0.
\end{cases}
\end{align}
Moreover, \cite[Theorem 2.5]{Sal} states that we can assume the complex structure on $J$ to act on $\mathfrak{g}^*$ either as
\[
(I) \quad Je^1=-e^2,\quad Je^2=e^1,\quad Je^3=-e^4,\quad  Je^4=e^3,\quad Je^5=-e^6,\quad Je^6=e^5
\]
or
\[
(II)\quad Je^1=-e^2,\quad Je^2=e^1,\quad Je^3=-e^6,\quad  Je^4=-e^5,\quad Je^5=e^4,\quad Je^6=e^3,
\]
so that the bases of $(\mathfrak{g}_{\C}^*)^{1,0}$ are, respectively,
\[
(I) \quad \eta^1=e^1+ie^2, \quad \eta^2=e^3+ie^4, \quad \eta^3=e^5+ie^6,
\]
and
\[
(II) \quad \eta^1=e^1+ie^2, \quad \eta^2=e^3+ie^6, \quad \eta^3=e^4+ie^5.
\]
For each family of complex structures, we proceed by imposing the integrability of $J$, that is,  by requiring that $(d\eta^j)^{0,2}=0$, for any $j\in\{1,2,3\}$. We analyze each case separately.

\subsection{Family $(I)$}
Starting from family $(I)$, we clearly have that $d\eta^1=d e^1+i de^2=0$. From $(d\eta^2)^{0,2}=0$, we obtain that $B_2=B_3=0$, whereas $(d\eta^3)^{0,2}=0$ implies that $D_4=D_7=0$ and
\[
C_2=C_5+D_3+D_5, \qquad D_6=C_3+C_4+D_2.
\]
Therefore, plugging these relations in \eqref{eq:Jacobi_nil3}, the complex structure equations of the basis $\{\eta^1,\eta^2,\eta^3\}$ boil down to
\[
\begin{cases}
d\eta^1&=0,\vspace{0.2cm}\\
d\eta^2&=\frac{i}{2}(A_1+iB_1)\eta^{1\overline{1}}\vspace{0.2cm}\\
d\eta^3&=\frac{1}{2} (D_3- iC_4 + D_5-iC_3)\eta^{12} + \frac{i}{2}(C_1 +i D_1)\eta^{1\overline{1}} + \frac{1}{2}(C_5 + iD_2 +  D_5+ i C_3 )\eta^{1\overline{2}}\vspace{0.2cm}\\
&+ \frac{1}{2}(-C_5-iD_2-D_3 -iC_4)\eta^{2\overline{1}} + \frac{i}{2}(C_6 +iD_8)\eta^{2\overline{2}},
\end{cases}
\]
with coefficients satisfying
\begin{equation}\label{eq:Jacobi_nil3_2}
B_1C_6=0,\quad A_1C_6=0,\quad B_1D_8=0, \quad A_1D_8=0.
\end{equation}
By setting
\begin{align*}
\lambda_1:=A_1+iB_1,\quad  \lambda_2:=C_1+iD_1,\quad \lambda_3:=D_3+iC_4,\\
\lambda_4:=D_5+iC_3,\quad \lambda_5:=C_5+iD_2,\quad \lambda_6:=C_6+iD_8,
\end{align*}
it turns out that
\begin{align*}
\begin{cases}
d\eta^1&=0,\vspace{0.2cm}\\
d\eta^2&=\frac{i}{2}\lambda_1\eta^{1\overline{1}}\vspace{0.2cm}\\
d\eta^3&=\frac{1}{2} (\overline{\lambda}_3+\overline{\lambda}_4)\eta^{12} + \frac{i}{2}\lambda_2\eta^{1\overline{1}} + \frac{1}{2}(\lambda_4+\lambda_5)\eta^{1\overline{2}}+ \frac{1}{2}(-\lambda_3-\lambda_5)\eta^{2\overline{1}} + \frac{i}{2}\lambda_6\eta^{2\overline{2}},
\end{cases}
\end{align*}
and from \eqref{eq:Jacobi_nil3_2}, we have that $\lambda_1\lambda_6=0$. As a result, the complex structures $J$ of family $(I)$ are nilpotent. We rewrite again the complex structure equations as
\begin{equation}\label{eq:struct_eq1}
\begin{cases}
d\eta^1=0,\\
d\eta^2=\lambda_1\eta^{1\overline{1}},\\
d\eta^3=\lambda_2\eta^{12}+\lambda_3\eta^{1\overline{1}}+\lambda_4\eta^{1\overline{2}}+\lambda_5\eta^{2\overline{1}}+\lambda_6\eta^{2\overline{2}},
\end{cases}
\end{equation}
with $\lambda_j\in\q[i]$, $j\in\{1,\dots,6\}$ and $\lambda_1\lambda_6=0$. Note that for any fixed Hermitian metric on $(\mathfrak{g},J)$, we can always apply the Gram-Schmidt process to the basis $\{\eta^1, \eta^2,\eta^3\}$ to obtain a orthonormal basis which still satisfies structure equations \eqref{eq:struct_eq1}; we will denote the orthonormal coframe with the same notation. In this framework, the fixed metric $g$ is diagonal and $\omega=\frac{i}{2}(\eta^{1\overline{1}}+\eta^{2\overline{2}}+\eta^{3\overline{3}}).$

We analyze separately the cases $\lambda_1=0$ and $\lambda_6=0$.

\subsubsection{Case $\lambda_1=0$}
As a consequence, structure equations \eqref{eq:struct_eq1} become
\[
\begin{cases}
d\eta^1=0,\quad 
d\eta^2=0,\\
d\eta^3=\lambda_2\eta^{12}+\lambda_3\eta^{1\overline{1}}+\lambda_4\eta^{1\overline{2}}+\lambda_5\eta^{2\overline{1}}+\lambda_6\eta^{2\overline{2}}.
\end{cases}
\]
The SKT condition for $g$, i.e., $\del\delbar\omega=0$, reads
\[
2i\del\delbar\omega=\del\delbar\eta^{3\overline{3}}=\left(-|\lambda_2|^2-|\lambda_4|^2-|\lambda_5|^2+2\mathfrak{Re}(\overline{\lambda}_3\lambda_6)\right)\eta^{12\overline{12}}=0.
\]

Set $C:=-|\lambda_2|^2-|\lambda_4|^2-|\lambda_5|^2+2\mathfrak{Re}(\overline{\lambda}_3\lambda_6)$. Since $(M,J)$ is not SKT, the metric $g$ cannot be SKT on $(\mathfrak{g},J)$ and it must hold that $C\neq 0$.

Then, consider the $ABC$-Massey product on $(\mathfrak{g},J)$ given by
\[
\mathcal{P}:=\left\langle[\eta^{1\overline{1}}]_{BC},[\eta^{2\overline{2}}]_{BC}, [\eta^{2\overline{2}}]_{BC} \right\rangle_{ABC}\in\frac{H_A^{2,2}(\mathfrak{g})}{[\eta^{1\overline{1}}]\cup H_A^{1,1}(\mathfrak{g})+[\eta^{2\overline{2}}]\cup H_A^{1,1}(\mathfrak{g})}.
\]
The forms $\eta^{1\overline{1}}$ and $\eta^{2\overline{2}}$ are both $d$-closed and $\del\delbar\ast$-closed, so they are Bott-Chern harmonic, hence $\mathcal{P}$ is well defined. Furthermore, since 
\[
\eta^{1\overline{1}}\wedge\eta^{2\overline{2}}=\del\delbar\left(-\frac{1}{C}\eta^{3\overline{3}}\right)\neq 0,
\]
if one sets $\tilde{\gamma}:=1$, by Theorem \ref{thm:obstruction} of Section \ref{sec:4} we obtain a non vanishing $ABC$-Massey product on $\mathfrak{g}$.

\subsubsection{$\lambda_6=0$} Then, the structure equations \eqref{eq:struct_eq} become
\begin{equation*}
\begin{cases}
d\eta^1=0,\\ d\eta^2=\lambda_1\eta^{1\overline{1}},\\
d\eta^3=\lambda_2\eta^{12}+\lambda_3\eta^{1\overline{1}}+\lambda_4\eta^{1\overline{2}}+\lambda_5\eta^{2\overline{1}}.
\end{cases}
\end{equation*}
The SKT condition for the fixed Hermitian metric $g$ reads
\[
2i\del\delbar\omega=\left(-|\lambda_2|^2-|\lambda_4|^2-|\lambda_5|^2\right)\eta^{12\overline{12}}=0.
\]
Since $(M,J)$ is not SKT, it must hold that $C:=-(|\lambda_2|^2+|\lambda_4|^2+|\lambda_5|^2)\neq 0$. Then, let us consider the $ABC$-Massey product on $\mathfrak{g}$ given by
\[
\mathcal{P}:=\langle[\eta^{1\overline{1}}]_{BC},[\eta^{2\overline{2}}]_{BC},[\eta^{1\overline{2}}]_{BC}\rangle_{ABC}\in\frac{H_A^{2,2}(\mathfrak{g})}{[\eta^{1\overline{1}}]\cup H_A^{1,1}(\mathfrak{g})+[\eta^{1\overline{2}}]\cup H_A^{1,1}(\mathfrak{g})}.
\]
The forms $\eta^{1\overline{1}}$, $\eta^{2\overline{2}}$ and $\eta^{1\overline{2}}$ are both $d$-closed and $\del\delbar\ast$-closed. Hence, they are Bott-Chern harmonic and $\mathcal{P}$ is well defined. Moreover, $\del\delbar(-\frac{1}{C}\eta^{3\overline{3}})=\eta^{1\overline{1}2\overline{2}}$, hence $\mathcal{P}$ is represented by the Aeppli cohomology class
\[
\left[-\frac{1}{C}\eta^{13\overline{23}}\right]_A.
\]
Note that the representative $-\frac{1}{C}\eta^{13\overline{23}}$ is also $d\ast$-closed, so that it is Aeppli harmonic and the cohomology class $[-\frac{1}{C}\eta^{13\overline{23}}]_A\neq 0$. It remains to prove that $$\left[\-\frac{1}{C}\eta^{13\overline{23}}\right]_A\notin [\eta^{1\overline{1}}]_{BC}\cup H_A^{1,1}+[\eta^{1\overline{2}}]_{BC}\cup H_A^{1,1}.$$ By contradiction, let us assume on the contrary that there exist $\nu_j, \mu_j\in\C$ and forms $R\in\bigwedge^{1,2}\mathfrak{g}_\C^*$, $S\in\bigwedge^{2,1}\mathfrak{g}_\C*$ such that
\begin{equation}\label{eq:ideal_nil_dim3}
-\frac{1}{C}\eta^{13\overline{23}}=\sum_{j=1}^{h_A^{1,1}}(\nu_j\eta^{1\overline{1}}+\mu_j\eta^{1\overline{2}})\wedge\xi^j+\del R+\delbar S.
\end{equation}
We now consider extension to $(M,J)$ of the differential forms and of the metric so far introduced on $(\mathfrak{g},J)$. By taking the $L^2(M)$-product of both sides of equation \eqref{eq:ideal_nil_dim3} with the form $\eta^{13\overline{23}}$, we obtain
\begin{align}\label{eq:ideal_2}
0\neq -\frac{1}{C}||\eta^{13\overline{23}}||^2=\int_M\sum_{j=1}\mu_j\eta^{1\overline{2}}\wedge\xi^j\wedge\ast(\eta^{13\overline{23}})+\int_M \del R\wedge\ast(\eta^{13\overline{23}})+\int_M \delbar S\wedge \ast(\eta^{13\overline{23}}),
\end{align}
with $\{\xi^{j}\}_j$ a basis of Aeppli harmonic $(1,1)$ forms on $\mathfrak{g}$. 
Since $\ast(\eta^{13\overline{23}})$ is both $\del$-closed and $\delbar$-closed, via Stokes' theorem the last two terms of the right-hand side of \eqref{eq:ideal_2} vanish, so to obtain
\begin{equation*}
0 \neq -\sum_{j=1}^{h_A^{1,1}}\int_M\mu_j\eta^{12\overline{12}}\wedge\xi^j,
\end{equation*}
which implies that there exists $j_0\in\{1,\dots, h_A^{1,1}\}$ such that $\xi^{j_0}=A\eta^{3\overline{3}}+\zeta$, for a non zero constant $A\in\C$ and a form $\zeta\in\bigwedge^{1,1}\mathfrak{g}_{\C}^*$ with $\zeta\not\in \mathcal{I}(\eta^3,\overline{\eta}^3)$. However, since $\xi^{j_0}\in\mathcal{H}_A^{1,1}(M)$, it must hold that
\[
0=\del\delbar\xi^{j_0}=A\del\delbar(\eta^{3\overline{3}})+\del\delbar\zeta=AC\eta^{12\overline{12}}\neq 0,
\]
with $C\in\C\setminus\{0\}$, which leads to a contradiction. Therefore, we have obtained that $$\left[\frac{1}{C}\eta^{13\overline{23}}\right]_A\notin [\eta^{1\overline{1}}]_{BC}\cup H_A^{1,1}+[\eta^{1\overline{2}}]\cup H_A^{1,1},$$ i.e., the $ABC$-Massey product $\mathcal{P}$ is not vanishing.

As a consequence, in both cases $\lambda_1=0$ and $\lambda_6=0$, the nilmanifolds endowed with non SKT invariant complex structures of Family $(I)$ admit a non vanishing $ABC$-Massey product by Lemma \ref{lemma:corrinvABC} and, hence, they are not geometrically-$BC$-formal.
\subsection{Family $(II)$}
Imposing the integrability of $J$, we obtain that $(d\eta^2)^{0,2}=0$ if and only if $C_3=C_5=C_6=0$, $C_4=B_2$, and $C_2=-B_3$. Moreover, $(d\eta^3)^{0,2}=0$ if and only if $D_2=D_5=D_8=0$, $D_7=D_3$, and $D_6=-D_4$. Hence, plugging these conditions in \eqref{eq:Jacobi_nil3}, we obtain the complex structure equations
\begin{align*}
\begin{cases}
d\eta^1=0,\\
d\eta^2=\frac{1}{2}(B_2 - iB_3)\eta^{13} + \frac{i}{2}(B_1 + iC_1)\eta^{1\overline{1}} + \frac{1}{2}(B_2 - iB_3)\eta^{1\overline{3}}\\
d\eta^3=\frac{i}{2}(A_1 +i D_1)\eta^{1\overline{1}} + \frac{i}{2}(D_3 +i D_4)\eta^{1\overline{2}} - \frac{i}{2}(D_3-i D_4)\eta^{2\overline{1}},
\end{cases}
\end{align*}
with coefficients satisfying
\begin{equation}\label{eq:Jacobi_nil3_3}
B_3D_3+B_2D_4=0.
\end{equation}
By setting
\[
\lambda_1=\frac{1}{2}(B_2-iB_3), \quad \lambda_2=\frac{i}{2}(B_1+iC_1), \quad \lambda_3=\frac{i}{2}(A_1+iD_1),\quad  \lambda_4=\frac{i}{2}(D_3+iD_4),
\]
the complex structure equations become
\begin{equation*}
\begin{cases}
d\eta^1=0,\\
d\eta^2=\lambda_1\eta^{13}+\lambda_2\eta^{1\overline{1}}+\lambda_1\eta^{1\overline{3}},\\
d\eta^3=\lambda_3\eta^{1\overline{1}}+\lambda_4\eta^{1\overline{2}}+\overline{\lambda}_4\eta^{2\overline{1}},
\end{cases}
\end{equation*}
with $\mathfrak{Re}(\lambda_1)\mathfrak{Re}(\lambda_4)+\mathfrak{Im}(\lambda_1)\mathfrak{Im}(\lambda_4)=0$.

Let us consider any Hermitian metric $g$ on $(\mathfrak{g},J)$ whose fundamental form $\omega$ has expression
\[
\omega=\frac{i}{2}\sum_{k=1}^3\omega_{k\overline{k}}\eta^{k\overline{k}}+\frac{1}{2}\sum_{1\leq j\leq k\leq 3}(\omega_{j\overline{k}}\eta^{j\overline{k}}-\overline{\omega}_{j\overline{k}}\eta^{k\overline{j}}),
\]
with $\{\omega_{j\overline{k}}\}_{j,k=1}^3$ a positive definite Hermitian matrix. By straightforward computations, it turns out that
\[
\del\delbar\omega=-\frac{i}{4}|\lambda_4|^2\omega_{3\overline{3}}\eta^{12\overline{12}}-\frac{i}{4}|\lambda_1|^2\omega_{2\overline{2}}\eta^{13\overline{13}}.
\]
Hence, the Hermitian metric $g$ is SKT if and only if $\lambda_1=0$ and $\lambda_4=0$. Since $(M,J)$ does not admit any SKT metrics, it holds that  $(\lambda_1,\lambda_4)\neq (0,0)$. We will consider two cases.

\subsubsection{Case $\lambda_4\neq 0$} We have that $\del\delbar\eta^{3\overline{3}}=-\frac{1}{2}|\lambda_4|^2\eta^{12\overline{12}}$ and the Bott-Chern cohomology classes $[\eta^{12}]\in H_{BC}^{2,0}(\mathfrak{g})$ and $[\eta^{\overline{12}}]\in H_{BC}^{0,2}(\mathfrak{g})$ are well defined. Then, the following $ABC$-Massey product on $(\mathfrak{g},J)$ is well defined
\[
\left\langle[\eta^{12}],[\eta^{\overline{12}}],[\eta^{\overline{12}}]\right\rangle_{ABC}=\left[\frac{2}{|\lambda_4|^2}\eta^{3\overline{3}\overline{12}}\right]_A\in\frac{H_A^{2,2}(\mathfrak{g})}{[\eta^{\overline{12}}]\cup H_A^{1,1}(\mathfrak{g})}.
\]
Moreover, $\text{Im}\,\delbar\restrict{\bigwedge^{1,2}\mathfrak{g}_\C^*}+\text{Im}\,\del\restrict{\bigwedge^{0,3}\mathfrak{g}_\C^*}\subset\text{Span}\langle\eta^{1\overline{123}},\eta^{2\overline{123}} \rangle$, hence $[\frac{2}{|\lambda_4|^2}\eta^{3\overline{123}}]_A\neq 0$ as an Aeppli cohomology class.

Suppose that $\left[\frac{2}{|\lambda_4|^2}\eta^{3\overline{123}}\right]_A\in [\eta^{\overline{12}}]\cup H_A^{1,1}(\mathfrak{g})$, i.e., there exist $\mu_j\in\C$, $R\in\bigwedge^{0,3}\mathfrak{g}_\C^*$, and $S\in\bigwedge^{1,2}\mathfrak{g}_\C^*$ such that
\begin{equation}\label{eq:ideal_nil3_2}
\frac{2}{|\lambda_4|^2}\eta^{3\overline{123}}=\sum_{j=1}^{h_A^{1,1}}\mu_j\eta^{\overline{12}}\wedge\xi^j+\del R+\delbar S,
\end{equation}
where $\{\xi^j\}_{j=1}^{h_A^{1,1}}$ is a basis of $\mathcal{H}_{\Delta_A}^{1,1}(\mathfrak{g})$. We take the usual extension to $(M,J)$ of the forms and the metric on $(\mathfrak{g},J)$. By multiplying by $\eta^{12}$ and integrating over $M$ each side of equation \eqref{eq:ideal_nil3_2}, we obtain
\[
0\neq \frac{2}{|\lambda_4|^2}\text{Vol}_M=\int_{M}\sum_{j=1}^{h_A^{1,1}}\mu_j\eta^{12\overline{12}}\wedge\xi^j+\int_M\del R\wedge\eta^{12}+\int_M\delbar S\wedge\eta^{12}.
\]
Since $d\eta^{12}=0$, the last two terms on the right hand side of the equation vanish by Stokes' theorem, yielding
\[
0\neq \int_{M}\sum_{j=1}^{h_A^{1,1}}\mu_j\eta^{12\overline{12}}\wedge\xi^j.
\]

Now, this is true if and only if there exists $j_0\in\{1,\dots, h_A^{1,1}\}$ such that $\xi^{j_0}=\alpha+A\eta^{3\overline{3}}\in\mathcal{H}_{\Delta_A}^{1,1}(\mathfrak{g})$, for some $\alpha\in\bigwedge^{1,1}\mathfrak{g}_\C^*$ and $0\neq A\in\C$. 
However, it is easy to see that $$\text{Ker}\del\delbar\restrict{\bigwedge^{1,1}\mathfrak{g}_{\C}^*}=\text{Span}\langle\eta^{1\overline{1}},\eta^{1\overline{2}},\eta^{1\overline{3}},\eta^{2\overline{1}},\eta^{2\overline{3}},\eta^{3\overline{1}},\eta^{3\overline{2}}\rangle$$ and the only non $\del\delbar$-closed $(1,1)$-forms on $\mathfrak{g}$ are $\del\delbar\eta^{2\overline{2}}=-\frac{1}{2}|\lambda_1|^2\eta^{13\overline{13}}$ and $\del\delbar\eta^{3\overline{3}}=-\frac{1}{2}|\lambda_4|^2\eta^{12\overline{12}}$. Then,
\[
0=\del\delbar(\xi^{j_0})=\del\delbar(\alpha)+A\del\delbar\eta^{3\overline{3}}\neq 0.
\] Therefore, we have obtained a contradiction and the $ABC$-Massey product
\[
\left\langle[\eta^{12}],[\eta^{\overline{12}}],[\eta^{\overline{12}}]\right\rangle_{ABC}
\]
is non vanishing on $(\mathfrak{g},J)$.

\subsubsection{Case $\lambda_4=0$} Since $(\lambda_1,\lambda_4)\neq (0,0)$, it holds that $\lambda
_1\neq 0$ and the structure equations become
\begin{equation}*
\begin{cases}
d\eta^1=0,\\
d\eta^2=\lambda_1\eta^{13}+\lambda_2\eta^{1\overline{1}}+\lambda_1\eta^{1\overline{3}},\\
d\eta^3=\lambda_3\eta^{1\overline{1}}.
\end{cases}
\end{equation}*
By rearranging the terms of the basis, it is easy to see that the complex structure $J$ on $\mathfrak{g}$ is nilpotent.

In this case, $\del\delbar\eta^{2\overline{2}}=-\frac{1}{2}|\lambda_1|^2\eta^{13\overline{13}}$ and the Bott-Chern cohomology classes $[\eta^{13}]\in H_{BC}^{2,0}(\mathfrak{g})$, $[\eta^{\overline{13}}]\in H_{BC}^{0,2}(\mathfrak{g})$ are well defined. Then, the $ABC$-Massey product on $\mathfrak{g}$ given by
\[
\left\langle[\eta^{13}],[\eta^{\overline{13}}],[\eta^{\overline{13}}]\right\rangle_{ABC}=\left[\frac{2}{|\lambda_1|^2}\eta^{2\overline{213}}\right]_A\in\frac{H_{A}^{1,3}(\mathfrak{g})}{[\eta^{\overline{13}}]\cup H_A^{1,1}(\mathfrak{g})},
\]
is well defined. Since $\text{Im}\,\delbar\restrict{\bigwedge^{1,2}\mathfrak{g}_\C^*}+\text{Im}\,\del\restrict{\bigwedge^{0,3}\mathfrak{g}_\C^*}\subset\text{Span}\langle\eta^{1\overline{123}},\eta^{2\overline{123}} \rangle$, we have that that $[\frac{2}{|\lambda_1|^2}\eta^{2\overline{213}}]\neq 0$ as an Aeppli cohomology class. By observing that the form $\eta^{13}$ is $d$-closed and $$\text{Ker}\,\del\delbar\restrict{\bigwedge^{1,1}\mathfrak{g}_{\C^*}}=\text{Span}\langle\eta^{1\overline{1}},\eta^{1\overline{2}},\eta^{1\overline{3}},\eta^{2\overline{1}},\eta^{2\overline{3}},\eta^{3\overline{1}},\eta^{3\overline{2}},\eta^{3\overline{3}}\rangle,$$
arguments analogous to the previous case yield that the $ABC$-Massey product
\[
\left\langle[\eta^{13}],[\eta^{\overline{13}}],[\eta^{\overline{13}}]\right\rangle_{ABC}
\]
is non vanishing on $(\mathfrak{g},J)$.

As a result, by Lemma \ref{lemma:corrinvABC} the nilmanifolds endowed with non SKT invariant complex structures of Family $(II)$ admit a non vanishing $ABC$-Massey product and, hence, they are not geometrically-$BC$-formal.
\end{proof}

\begin{rem}
If $(\Gamma\backslash G=M,J)$ is either a $6$-dimensional nilmanifold $M$ with an SKT invariant complex structure $J$ or, more in general, a $2n$-dimensional nilmanifold endowed with a SKT special complex structure, then from Theorem \ref{thm:SKT-geomBC-dim3} (see also \cite[Theorem 7.2]{SfeTom2}) and \cite[Theorem 7.4]{SfeTom2} we know that $(M,J)$ is geometrically-$BC$-formal. Since $J$ is nipotent, by \cite{Ang}, the Aeppli cohomology and the Bott-Chern cohomology of $(M,J)$ are computed, respectively, by the space of invariant Aeppli (respectively, Bott-Chern) harmonic forms with respect to a fixed invariant SKT metric $g$ on $(M,J)$. Moreover, the space $\mathcal{H}:=\mathcal{H}_{\Delta_A}(\mathfrak{g})+\mathcal{H}_{\Delta_{BC}}(\mathfrak{g})$ is closed under $\star$ and since \cite[Lemma 7.3]{SfeTom2} $\del\delbar\equiv 0$ on left-invariant forms on $G$, then $\del\delbar\restrict{\mathcal{H}}\equiv 0$. Hence, any invariant SKT is, in particular, an \emph{ABC-geometrically formal metric} in the sense of \cite{MS} and $(M,J)$ is \emph{weakly formal} in the sense of \cite{MS}.
\end{rem}

\section{SKT metrics and Bott-Cher-geometrically formal metrics in higher dimensions}\label{sec:4}
In this section we further study the relation between SKT metrics and geometrically-$BC$-formal metrics for nilmanifolds of complex dimension strictly higher than $3$. 

As proved in \cite{SfeTom2}, for the class of nilmanifolds endowed with special type complex structures, in any complex dimension the existence of a SKT metric implies the existence of a geometrically-$BC$-formal metric. In Proposition \ref{prop:geomBC_SKT}, we prove a partial converse of this result by showing that, if a nilmanifold endowed with a special type complex structure admits no SKT metric, then it does not admit any invariant geometrically-$BC$-formal metric. However, this does not suffice to show that there exist no (not necessarily invariant) geometrically-$BC$-formal metrics. Thus, in Theorem \ref{thm:obstruction}, we give explicit sufficient conditions on the Bott-Chern cohomology for the existence of a non vanishing $ABC$-Massey product; hence, under such conditions, there cannot exist geometrically-$BC$-formal metrics on such a class of manifolds.

Finally, in Proposition \ref{prop:no_sp_type_SKT} we show that by dropping the hypothesis of special type complex structure, we can construct a nilmanifold with invariant nilpotent complex structure which is SKT but admits a non trivial $ABC$-Massey product.

Let now $(M,J)$ be a $2n$-dimensional nilmanifold endowed by a complex structure of special type, see \eqref{eq:struct_eq}. Then, the complex structure $J$ is integrable and a $2$-step nilpotent complex structure \cite{CorFerGraUga} on $M$.  By \cite{Ang}, the Aeppli cohomology and the Bott-Chern cohomology of $(M,J)$ are computed via the complex of left-invariant forms on $G$. As a consequence, by Lemma \ref{lemma:corrinvABC}, $ABC$-Massey products on $(\mathfrak{g},J)$ corresponds bijectively to $ABC$-Massey products on $(M,J)$.

Let $g$ be an invariant Hermitian metric on $(M,J)$ and denote by
\[
\omega=\frac{i}{2}\sum_{j=1}^k\omega_{j\overline{j}}\eta^{j\overline{j}}+\frac{1}{2}\sum_{1\leq i< j \leq k}(\omega_{i\overline{j}}\,\eta^{i\overline{j}}-\overline{\omega}_{i\overline{j}}\,\eta^{j\overline{i}}).
\]
its fundamental form.

For this family of nilmanifolds, by structure equations \eqref{eq:struct_eq}, the metric $g$ is SKT if and only if $\del\delbar\eta^{k\overline{k}}=0$, i.e., the SKT condition does not depend on the choice of the Hermitian metric.

In that case, it was shown in \cite{SfeTom2} that if $J$ is such that every invariant Hermitian metric on $(M,J)$ is SKT, the manifold $(M,J)$ is geometrically-$BC$-formal, hence \cite{TarTom} every $ABC$-Massey product vanishes.
\begin{rem}\label{rmk:SKT_ABC_prod_J_special}
As a consequence of \cite[Theorem 7.4]{SfeTom2}, a necessary condition for the existence of non vanishing triple $ABC$-Massey products is that $(M,J)$ does not admit any SKT metric. 
\end{rem}

Furthermore, we observe the following result.
\begin{prop}\label{prop:geomBC_SKT}
For nilmanifolds endowed with complex structure of special type, if there exists an invariant geometrically-$BC$-formal metric, then it is also SKT.
\end{prop}
\begin{proof}
Let us fix $\{\eta^1,\dots,\eta^k\}$ a coframe of invariant $(1,0)$-forms on $(M,J)$ satisfying \eqref{eq:struct_eq}. By Gram-Schmidt, we can assume that the coframe $\{\eta^1,\dots, \eta^k\}$ is unitary and still satisfy structure equations of type \eqref{eq:struct_eq}. In this situation, the fundamental form of $g$ can be written as $\omega=\frac{i}{2}\sum_{j=1}^k\eta^{j\overline{j}}$. Let us assume that $g$ is not SKT, i.e.,
\[
0\neq \del\delbar\eta^{k\overline{k}}=\sum_{1\leq r,u<s,v\leq k-1 }L_{rs\overline{u}\overline{v}}\eta^{rs\overline{u}\overline{v}}.
\]
We can then choose $r_0,s_0,u_0.v_0$ such that $L_{r_0s_0\overline{u}_0\overline{v}_0}\neq 0$. By structure equations and bidegree reasons, the forms $\alpha:=\eta^{r_0s_0}$ and $\beta:=\eta^{\overline{u}_0\overline{v}_0}$ are both  $d$-closed and $\del\delbar\ast$-closed, so they are Bott-Chern harmonic. We check whether this holds also for the wedge product $\alpha\wedge\beta$. By Leibniz rule, the form $\alpha\wedge\beta$ is $d$-closed. However, $\ast(\alpha\wedge\beta)=\varepsilon\eta^{1\dots r_0^* \dots s_0^* \dots k \overline{1}\dots \overline{u}_0^* \dots \overline{v}_0^*\dots \overline{k}}$, with $\varepsilon$ a sign constant, so that
\[
\del\delbar\ast(\alpha\wedge\beta)=(-1)^{k-3}\varepsilon\eta^{1\dots r_0^* \dots s_0^* \dots k-1 \overline{1}\dots \overline{u}_0^* \dots \overline{v}_0^*\dots \overline{k-1}}\wedge\del\delbar(\eta^{k\overline{k}})=\varepsilon'L_{r_0s_0\overline{u}_0\overline{v}_0}\eta^{1\dots k-1 \overline{1}\dots \overline{k-1}}\neq 0,
\]
with $\varepsilon'$ a sign constant. As a consequence, the form $\alpha\wedge\beta$ is not Bott-Chern harmonic, i.e., the invariant metric $g$ is not geometrically-$BC$-formal.
\end{proof}
We now prove a sufficient condition for the existence of a non trivial $ABC$-Massey product on nilmanifolds endowed with a special type complex structure.

Recall that, in our notation, a real $2k$-dimensional nilmanifold $M=\Gamma\backslash G$ admits a special type complex structure $J$ if it admits a coframe of invariant $(1,0)$-forms $\{\eta^1,\dots, \eta^k\}$ such that
\begin{equation}\label{eq:struct_eq_spec_type}
\begin{cases}
d\eta^j=0, \quad j\in \{1,\dots,k-1\}\\
d\eta^k\in\text{Span}_{\C}\langle \eta^{ij}, \eta^{i\overline{j}}\rangle_{i,j=1,\dots, k-1}
\end{cases}.
\end{equation}
Let $g$ be a fixed invariant Hermitian metric on $(M,J)$ (which we can always assume to be diagonal) and let $\omega=\frac{i}{2}\sum_{j=1}^k\eta^{j\overline{j}}$ be its fundamental form.

By Remark \ref{rmk:SKT_ABC_prod_J_special}, we assume that $(M,J)$ is not SKT, i.e., $\del\delbar\eta^{k\overline{k}}\neq 0$. In this case, the $\del\delbar$-operator on invariant forms is not zero only on $\mathcal{I}(\eta^{k\overline{k}})$. Moreover, structure equations \eqref{eq:struct_eq_spec_type} imply that
\begin{equation}\label{eq:im_deldelbar}
\im\del\delbar\subset \textstyle\bigwedge^{\bullet,\bullet}\left(\text{Span}_\C\langle \eta^{1},\dots, \eta^{k-1}, \eta^{\overline{1}}, \dots ,\eta^{\overline{k-1}}
\rangle \right).
\end{equation}
We recall that special type complex structures are in particular nilpotent, hence \cite[Theorem 2.8]{Ang} implies that the Bott-Chern and Aeppli cohomology of $(M,J)$ are computed via the complex of left-invariant forms on $G$, or equivalently, they are isomorphic to the Aeppli and Bott-Chern cohomology of the Lie algebra $\mathfrak{g}$ of $G$. In particular, by further exploiting the structure equations \eqref{eq:struct_eq_spec_type}, straightforward computations allow us to describe the structures of Bott-Chern and Aeppli cohomologies of $(M,J)$.
\begin{lemma}\label{lemma:BCcohom}
The Bott-Chern cohomology of $(M,J)$ of any bidegree $(p,q)$ admits the following decomposition
\begin{equation*}
H_{BC}^{p,q}(M)\cong \frac{H_1}{I_1}\oplus H_2 \oplus H_3,
\end{equation*}
where
\begin{align*}
H_1&=\Span_{\C} \langle \eta^{i_1\dots i_{p}}\wedge\eta^{\overline{j}_1\dots \overline{j}_{q}}\rangle_{i_l,j_l\ne k}\\
I_1&=\im\del\delbar\left(\Span_{\C} \langle\eta^{k\overline{k}}\wedge\eta^{i_1\dots i_{p-2}}\wedge\eta^{\overline{j}_1\dots \overline{j}_{q-2}}\rangle\right)\\
H_{2}&=\Ker d\cap \left(\Span_{\C}\langle  \eta^k\wedge\eta^{i_1\dots i_{p-1}}\wedge\eta^{\overline{j}_1 \dots \overline{j}_{q}}\rangle\oplus \Span_{\C}\langle \eta^{\overline{k}}\wedge\eta^{i_1\dots i_{p}}\wedge\eta^{\overline{j}_1 \dots \overline{j}_{q-1}} \rangle\right)\\
H_3 &= \Ker d\cap \Span_{\C}\langle \eta^{k\overline{k}}\wedge \eta^{i_1\dots i_{p-1}}\wedge\eta^{\overline{j}_1\dots \overline{j}_{q-1}}\rangle.
\end{align*}
In particular, if either $p<2$ or $q<2$, we have that
\[
H_{BC}^{p,q}(M)\cong H_1\oplus H_2 \oplus H_3.
\]
\end{lemma}
\begin{proof}
By \cite[Theorem 2.8]{Ang}, we have that $H_{BC}^{p,q}(M)\cong H_{BC}^{p,q}(\mathfrak{g})$, where \[
H_{BC}^{p,q}(\mathfrak{g})=\frac{\text{Ker}(d\colon \bigwedge^{p,q}\mathfrak{g}_\C^*\rightarrow\bigwedge^{p+1,q}\mathfrak{g}_\C^*\oplus \bigwedge^{p,q+1}\mathfrak{g}_\C^*)}{\text{Im}(\del\delbar\colon\bigwedge^{p-1,q-1}\mathfrak{g}_\C^*\rightarrow \bigwedge^{p,q}\mathfrak{g}_\C^*)}.
\]
But then the following decomposition holds
\[
\text{Ker}(\textstyle d\colon \bigwedge^{p,q}\mathfrak{g}_\C^*\rightarrow\bigwedge^{p+1,q}\mathfrak{g}_\C^*\oplus \bigwedge^{p,q+1}\mathfrak{g}_\C^*)=H_1\oplus H_2\oplus H_3.
\]
By \eqref{eq:im_deldelbar}, we have that
\[
\text{Im}(\del\delbar\colon\textstyle\bigwedge^{p-1,q-1}\mathfrak{g}_\C^*\rightarrow \bigwedge^{p,q}\mathfrak{g}_\C^*)=I_1.
\]
Then, since $I_1\cap H_2 =\{0\}$ and $I_1\cap H_3=\{0\}$, we can conclude that
\[
H_{BC}^{p,q}(\mathfrak{g})=\frac{H_1\oplus H_2\oplus H_3}{I_1}\cong \frac{H_1}{I_1}\oplus H_2\oplus H_3.
\]
Note that if either $p<2$ or $q<2$, then $I_1=\{0\}$, which concludes the proof.
\end{proof}
\begin{lemma}\label{lemma:aep_cohom}
The Aeppli cohomology of $(M,J)$ of any bidegree $(p,q)$ admits the following decomposition
\begin{equation*}
H_{A}^{p,q}(M)\cong \frac{K_1}{L_1}\oplus \frac{K_2}{L_2} \oplus K_3,
\end{equation*}
where
\begin{align*}
K_1=&\Span_{\C}\langle\eta^{i_1\dots i_p\overline{j}_1\dots \overline{j}_{q}}\rangle_{i_l,j_l\neq k}\\
L_1=&\im\del\left(\Span_{\C}\langle\eta^{k i_{1}\dots i_{p-2}\overline{j}_1\dots \overline{j}_{q}}\rangle\right) \oplus \im\delbar\left(\Span_{\C}\langle\eta^{k i_{1}\dots i_{p-1}\overline{j}_1\dots \overline{j}_{q-1}}\rangle\right)\\
&\oplus \im\del\left(\Span_{\C}\langle\eta^{\overline{k} i_{1}\dots i_{p-1}\overline{j}_1\dots \overline{j}_{q-1}}\rangle\right) \oplus \im\delbar\left(\Span_{\C}\langle\eta^{\overline{k} i_{1}\dots i_{p}\overline{j}_1\dots \overline{j}_{q-1}}\rangle\right) \\
K_2=&\Span_{\C}\langle\eta^{k i_{1}\dots i_{p-1}\overline{j}_{1}\dots \overline{j}_{q}},\eta^{\overline{k}i_1\dots i_p\overline{j}_1\dots \overline{j}_{q-1}}\rangle\\
L_2=&\im\del\left(\Span_{\C}\langle\eta^{k\overline{k} i_{1}\dots i_{p-2}\overline{j}_1\dots \overline{j}_{q-1}}\rangle\right)\oplus \im\delbar\left(\Span_{\C}\langle\eta^{k\overline{k}i_{1}\dots i_{p-1}\overline{j}_{1}\dots \overline{j}_{q-2}}\rangle\right)\\
K_3=&\Ker\del\delbar\cap\Span_{\C}\langle \eta^{k\overline{k}i_1\dots i_{p-1}\overline{j}_1\dots \overline{j}_{q-1}}\rangle.
\end{align*}
\begin{proof}
The proof is analogous to the proof of Lemma \ref{lemma:BCcohom}.
\end{proof}
\end{lemma}

We will now work at the level of the Lie algebra $\mathfrak{g}$ of the universal cover $G$ of $M$. We note that  $\mathfrak{g}$ is endowed with a special type complex structure inherited by $(M,J)$, which we will still denote  by $J$. Therefore, there exists a basis $\{\eta^1,\dots,\eta^k\}$ of $(1,0)$-forms on $(\mathfrak{g},J)$ satisfying \eqref{eq:struct_eq_spec_type} and the fixed invariant Hermitian metric $g$ on $(M,J)$ descends to a Hermitian metric on $(\mathfrak{g},J)$ with fundamental form $\omega=\frac{i}{2}\sum_{j=1}^k\eta^{j\overline{j}}$. 

Let us then consider a $ABC$-Massey product $\mathcal{P}$ on $(\mathfrak{g},J)$ given by the following expression
$$
\mathcal{P}=\Bigl\langle[\alpha], [\beta], [\beta]\Bigr\rangle_{ABC},
$$
where $[\alpha]\in H_{BC}^{p,q}(\mathfrak{g})$, $[\beta]\in H_{BC}^{r,s}(\mathfrak{g})$. We assume that the representatives $\alpha$ and $\beta$ satisfy the following properties
\[
\alpha= \eta^{i_1\dots i_p}\wedge\eta^{\overline{j}_1\dots \overline{j}_q}, \qquad \beta= \eta^{l_1\dots l_r}\wedge\eta^{\overline{m}_1\dots \overline{m}_s}, \qquad \alpha\wedge\beta\neq 0.
\]
In order for $\mathcal{P}$ to be a well defined non trivial $ABC$-Massey product, by \eqref{eq:im_deldelbar}, it must hold that
\begin{itemize}
\item
$\alpha, \beta\in H_1$ of Lemma \ref{lemma:BCcohom}, i.e., $$\alpha\in\textstyle \bigwedge^{p,q}\left(\text{Span}_\C\langle \eta^{1}, \dots, \eta^{k-1}, \eta^{\overline{1}}, \dots, \eta^{\overline{k-1}} \rangle\right), \quad \beta\in\bigwedge^{r,s}\left(\text{Span}_\C\langle \eta^{1}, \dots, \eta^{k-1}, \eta^{\overline{1}}, \dots, \eta^{\overline{k-1}} \rangle\right), $$
\item there exists a form $f=\eta^{k\overline{k}}\wedge\tilde{\gamma}\in \bigwedge^{p+r-1, q+s-1}\mathfrak{g}_\C^*$ with $\tilde{\gamma}\in\bigwedge^{p+r-2,q+s-2}\mathfrak{g}_\C^*$ such that
\[
(-1)^{p+q}\del\delbar f=\alpha\wedge\beta.
\]
\end{itemize}
We assume that $\tilde{\gamma}$ can be written as $\tilde{\gamma}:=\eta^{t_1\dots t_{p+r-2}}\wedge\eta^{\overline{v}_1\dots \overline{v}_{q+s-2}}$ and, up to a constant, it is the only $(p+r-2,q+s-2)$-form $\varphi$ which satisfies
$$
0\neq \del\delbar(\eta^{k\overline{k}})\wedge\varphi\in \Span_{\C}\langle\alpha\wedge\beta\rangle.
$$
Then the $ABC$-Massey product $\mathcal{P}$ is represented, up to constant, by the Aeppli cohomology class
\[
[\eta^{k\overline{k}}\wedge\tilde{\gamma}\wedge \beta]\in H_{A}^{p+2r-1, q+2s-1}(\mathfrak{g}).
\]
Note that since holds $\ast(\eta^{k\overline{k}}\wedge\tilde{\gamma}\wedge \beta)\not\in \mathcal{I}(\eta^k,\overline{\eta}^k)$, by structure equations $\ast(\eta^{k\overline{k}}\wedge\tilde{\gamma}\wedge \beta)$ is $\del$-closed and $\delbar$-closed; hence, the form $\eta^{k\overline{k}}\wedge\tilde{\gamma}\wedge \beta$ is Aeppli harmonic and, as a Aeppli cohomology class, $[\eta^{k\overline{k}}\wedge\tilde{\gamma}\wedge \beta]_A\neq 0$.

As a $ABC$-Massey product, $\mathcal{P}$ is non vanishing if and only if $[\eta^{k\overline{k}}\wedge\tilde{\gamma}\wedge \beta]_A\notin \mathcal{J}$, where $\mathcal{J}$ is the ideal
\[
[\alpha]_{BC}\cup H_A^{2r-1,2s-1}(\mathfrak{g})+[\beta]_{BC}\cup H_A^{p+r-1 ,q+s-1}(\mathfrak{g}).
\] 
Let us suppose by contradiction that this is the case, i.e., there exists $\lambda_j, \mu_l\in\C$ and $R\in\bigwedge^{p+2r-2,q+2s-1}\mathfrak{g}_\C^*$, $S\in\bigwedge^{p+2r-1, q+2s-2}\mathfrak{g}_\C^*$ such that
\begin{equation}\label{eq:ideal}
\eta^{k\overline{k}}\wedge\tilde{\gamma}\wedge \beta= \sum_{j=1}^{h_A^{2r-1, 2s-1}}\lambda_j\, \alpha \wedge\xi^j + \sum_{l=1}
^{h_A^{p+r-1,q+s-1}} \mu_l\,\beta \wedge\psi^l + \del R + \delbar S,
\end{equation}
with $\{\xi^j\}_{j=1}^{h_A^{2r-1,2s-1}}$ and $\{\psi^l\}_{l=1}^{h_A^{p+r-1,q+s-1}}$ bases of $\mathcal{H}_A^{2r-1,2s-1}(\mathfrak{g})$ and, respectively, of $\mathcal{H}_A^{p+r-1,q+s-1}(\mathfrak{g})$.

By taking the $L^2$-product of equation \eqref{eq:ideal} with the left hand side, we obtain
\begin{align*}
0\neq ||\eta^{k\overline{k}}\wedge\tilde{\gamma}\wedge \beta||_{L^2}^2 &= \int_M \left(\sum_j \lambda_j\alpha\wedge\xi^j\right)\wedge\ast(\eta^{k\overline{k}}\wedge\tilde{\gamma}\wedge \beta)\\
&+\int_M \left(\sum_l \mu_l\beta\wedge\psi^l\right)\wedge\ast(\eta^{k\overline{k}}\wedge\tilde{\gamma}\wedge \beta)\\
&+\int_M \del R\wedge\ast (\eta^{k\overline{k}}\wedge\tilde{\gamma}\wedge \beta) + \int_M \delbar S \wedge \ast(\eta^{k\overline{k}}\wedge\tilde{\gamma}\wedge \beta).
\end{align*}
Being $\ast\eta^{k\overline{k}}\wedge\tilde{\gamma}\wedge\beta$ a $d$-closed form, by Stokes' theorem the last two terms vanish, hence yielding
\[
0 \neq \int_M \left(\sum_j \lambda_j\alpha\wedge\xi^j\right)\wedge\ast(\eta^{k\overline{k}}\wedge\tilde{\gamma}\wedge \beta) + \int_M \left(\sum_l \mu_l\beta\wedge\psi^l\right)\wedge\ast(\eta^{k\overline{k}}\wedge\tilde{\gamma}\wedge \beta).
\]
Then, it must exist $j_0\in\{1,\dots, h_A^{2r-1,2s-1}\}$ (or $l_0\in\{1,\dots, h_A^{p+r-1,q+s-1}\}$), such that either $$\alpha\wedge\xi^{j_0}\wedge\ast(\eta^{k\overline{k}}\wedge\tilde{\gamma}\wedge\beta)\neq 0$$ or $$\beta\wedge\psi^{l_0}\wedge\ast(\eta^{k\overline{k}}\wedge\tilde{\gamma}\wedge\beta)\neq 0.$$

Suppose the first case holds. By definition of Hodge $\ast$-operator, this is equivalent to
\[
g(\alpha\wedge\xi^{j_0},\eta^{k\overline{k}}\wedge\tilde{\gamma}\wedge\beta)\neq 0
\]
Since the product $g$ is diagonal with respect to the basis $\{\eta^1,\dots, \eta^k\}$, by Lemma \ref{lemma:aep_cohom} this forces $\xi^{j_0}\in K_3$, i.e., $\xi^{j_0}=\eta^{k\overline{k}}\wedge\tilde{\omega}$, with $\tilde{\omega}\in \bigwedge^{2r-2,2s-2}\mathfrak{g}_\C^*$, which yields
\[
g(\eta^{k\overline{k}}\wedge\alpha\wedge\tilde{\omega},\eta^{k\overline{k}}\wedge\tilde{\gamma}\wedge\beta) \neq 0.
\]
Since $\alpha\wedge\beta\neq 0$ and $\alpha\not\in\mathcal{I}(\eta^k,\overline{\eta}^k)$, then $\gamma=\alpha\wedge\zeta$, for some $\zeta\in\bigwedge^{r-2,s-2}\mathfrak{g}_\C^*$. Note that if either $r<2$ or $s<2$, we obtain a contradiction. Let us assume the opposite, i.e., $r,s\geq 2$. Then,
\[
(-1)^{p+q}\alpha\wedge\beta=\del\delbar(\eta^{k\overline{k}})\wedge \tilde{\gamma}=\del\delbar(\eta^{k\overline{k}})\wedge\alpha\wedge\zeta,
\]
which implies that $\beta=\del\delbar\left((-1)^{p+q+1}\eta^{k\overline{k}}\wedge\zeta\right)$, which leads to a contradiction, since $\beta$ is a Bott-Chern harmonic form, i.e., $\beta$ cannot be $\del\delbar$-exact.

On the other hand, let us assume that $\beta\wedge\psi^{l_0}\wedge\ast(\eta^{k\overline{k}}\wedge\tilde{\gamma}\wedge\beta)\neq 0$, which is equivalent to
\[
g(\beta\wedge\psi^{l_0},\eta^{k\overline{k}}\wedge\tilde{\gamma}\wedge\beta) \neq 0.
\]
Since $\beta$ does not contain $\eta^{k\overline{k}}$, again by Lemma \ref{lemma:aep_cohom} the form $\psi^{l_0}\in K_3$, i.e., $\psi^{l_0}=\eta^{k\overline{k}}\wedge \tilde{\omega}$, with $\tilde{\omega}=\tilde{\omega}_1+\dots +\tilde{\omega}_d\in\bigwedge^{p+r-2,q+s-2}\mathfrak{g}_{\C}^*$ and each form $\tilde{\omega_i}=C_i\eta^{I\overline{J}}$ for $C_i\in\C$ and $I,J\subset\{1,\dots, k\}$, $|I|=p+r-2, |J|=q+s-2$. We obtain
\[
0 \neq g(\beta\wedge\eta^{k\overline{k}}\wedge\tilde{\omega},\eta^{k\overline{k}}\wedge\tilde{\gamma}\wedge\beta)=(-1)^{|\beta|\cdot |\tilde{\omega}|}g(\tilde{\omega},\tilde{\gamma})
\]
which then implies that there exists $1 \leq d_0\leq d$ and $0 \neq C\in\C$ such that $\tilde{\omega}_{d_0}=C\tilde{\gamma}$. So, we have that
\[
0=\del\delbar\psi^{l_0}=\del\delbar(\eta^{k\overline{k}})\wedge(\tilde{\omega}_1+\dots +\tilde{\omega}_d-\tilde{\omega}_{d_0})+\del\delbar(\eta^{k\overline{k}})\wedge \tilde{\omega}_{d_0}=\del\delbar(\eta^{k\overline{k}})\wedge (\tilde{\omega}-\tilde{\omega}_{d_0})+(-1)^{p+q}C\alpha\wedge\beta.
\]
Consequently, $-\frac{1}{C}\eta^{k\overline{k}}\wedge(\tilde{\omega}-\tilde{\gamma})$ is a $\del\delbar$-primitive of $(-1)^{p+q}\alpha\wedge\beta$. However, by unicity of the $\del\delbar$-primitive $\tilde{\gamma}$, we obtain a contradiction.

As a result, the Aeppli cohomology class $[\eta^{k\overline{k}}\wedge\tilde{\gamma}\wedge\beta]\notin \mathcal{J}$, hence the $ABC$-Massey product on $(\mathfrak{g},J)$
\[
\mathcal{P}=\Bigl\langle[\alpha],[\beta],[\beta]\Bigr\rangle_{ABC}\in\frac{H_{A}^{p+2r-1, q+2s-1}(\mathfrak{g})}{[\alpha]\cup H_A^{2r-1,2s-1}(\mathfrak{g})+[\beta]\cup H_A^{p+r-1 ,q+s-1}(\mathfrak{g})}
\]
is not vanishing. Finally, by Lemma \ref{lemma:corrinvABC}, we have that $\mathcal{P}$ corresponds to a non trivial $ABC$-Massey product on $(M,J)$.

To summarize, we have proved the following.
\begin{theorem}\label{thm:obstruction}
Let $(M=\Gamma\backslash G,J)$ be a nilmanifold of complex dimension $k$ with a invariant complex structure $J$ determined by a basis $\{\eta^1,\dots, \eta^{k}\}$ of $(1,0)$-forms such that
\[
\begin{cases}
d\eta^j=0, \qquad j\in\{1,\dots, k-1\},\\
d\eta^k\in \Span_{\C}\langle\eta^{jl},\eta^{j\overline{l}}\rangle_{j,l=1}^{k-1}.
\end{cases}
\]
If there exist forms $\alpha:=\eta^{i_1\dots i_p\overline{j}_1\dots\overline{j}_q}\in\bigwedge^{p,q}\mathfrak{g}_\C^*$, $\beta:=\eta^{l_1\dots l_{r}\overline{m}_1\dots \overline{m}_s}\in\bigwedge^{r,s}\mathfrak{g}_\C^*$ and a unique form (up to constant) $\tilde{\gamma}=\eta^{t_1\dots t_{p+r-2}\overline{v}_1\dots \overline{v}_{q+s-2}}\in\bigwedge^{p+r-2,q+s-2}\mathfrak{g}_\C^*$, such that
\begin{enumerate}
\item the forms $\alpha$ and $\beta$ are Bott-Chern harmonic and $\alpha\wedge\beta\neq 0$,
\item the form $\eta^{k\overline{k}}\wedge\tilde{\gamma}$ is a $\del\delbar$-primitive of $\alpha\wedge\beta$, i.e., $(-1)^{p+q}\alpha\wedge\beta=\del\delbar(\eta^{k\overline{k}})\wedge\tilde{\gamma}$, and $\tilde{\gamma}\wedge\beta\neq 0$,
\end{enumerate} 
then $(M,J)$ admits a non vanishing $ABC$-Massey product given by
\[
\Bigl\langle[\alpha]_{BC},[\beta]_{BC},[\beta]_{BC}\Bigr\rangle_{ABC}.
\]
\end{theorem}

As an immediate application of Theorem \ref{thm:obstruction} we provide the explicit examples of two families of nilmanifolds admitting non trivial $ABC$-Massey products.
\begin{ex}\label{example:1}
Consider the family of nilmanifolds $(\Gamma\backslash G= M,J)$ endowed with invariant complex structure $J$ characterized by a basis $\{\eta^1,\eta^2,\eta^3, \eta^4\}$ of left-invariant $(1,0)$-forms on $G$ such that 
\begin{align*}
\begin{cases}
d\eta^1=d\eta^2=d\eta^3=0,\\
d\eta^4=A\eta^{12}+B\eta^{1\overline{3}}+C\eta^{2\overline{1}},
\end{cases}
\end{align*}
with $A,B,C\in\q[i]\setminus\{0\}$. Note that
\[
0\neq \del\delbar(\eta^{4\overline{4}})=-(|A|^2+|C|^2)\eta^{12\overline{12}}-|B|^2\eta^{13\overline{13}}.
\]
Then
\[
\Bigl\langle[\eta^{12\overline{2}}]_{BC},[\eta^{3\overline{1}}]_{BC},[\eta^{3\overline{1}}]_{BC}\Bigr\rangle_{ABC}
\]
is a non vanishing triple $ABC$-Massey product. It suffices to set $\tilde{\gamma}=-\frac{1}{|B|^2}\eta^2$ and then Theorem \ref{thm:obstruction} directly applies.
\end{ex}
\begin{ex}\label{example:2}
Consider the family of nilmanifolds $(\Gamma\backslash G= M,J)$ endowed with invariant complex structure $J$ characterized by a basis $\{\eta^1,\eta^2,\eta^3, \eta^4, \eta^5\}$ of left-invariant $(1,0)$-forms on $G$ such that
\begin{align*}
\begin{cases}
d\eta^1=d\eta^2=d\eta^3=d\eta^4=0,\\
d\eta^5=D\eta^{12}+E\eta^{3\overline{2}}+F\eta^{4\overline{3}},
\end{cases}
\end{align*}
with $E,F,G\in\q[i]\setminus \{0\}$. Note that
\[
0\neq \del\delbar(\eta^{5\overline{5}})=-|D|^2\eta^{12\overline{12}}-|E|^2\eta^{23\overline{23}}-|F|^2\eta^{34\overline{34}}.
\]
Then,
\[
\Bigl\langle[\eta^{34\overline{1}}]_{BC}, [\eta^{1\overline{12}}]_{BC}, [\eta^{1\overline{12}}]_{BC}\Bigr\rangle_{ABC}
\]
is a non vanishing triple $ABC$-Massey product. In this case, it suffices to set $\tilde{\gamma}=-\frac{1}{|D|^2}\eta^{34}$ and Theorem \ref{thm:obstruction} directly follows.
\end{ex}
However, by dropping the hypotesis of special type complex structure, we show that there exist families of nilmanifolds endowed with nilpotent complex structures which admit SKT metrics but are never geometrically-$BC$-formal.

\begin{prop}\label{prop:no_sp_type_SKT}
There exist SKT nilmanifolds with nilpotent complex structures admitting no geometrically-$BC$-formal metrics.
\end{prop}
\begin{proof}
Let $(M=\Gamma\backslash G,J)$ be a family of nilmanifolds of complex dimension $4$ endowed with an invariant complex structure $J$ characterized by a basis $\{\eta^1,\eta^2, \eta^3, \eta^4\}$ of $(\mathfrak{g}^{(1,0)})^*$ such that
\begin{equation}\label{eq:struct_eq_SKT_non_geom_BC}
\begin{cases}
d\eta^1=0\\
d\eta^2=0\\
d\eta^3=B_1\eta^{2\overline{1}}\\
d\eta^4=G_1\eta^{12}+D_1\eta^{1\overline{1}}+D_2\eta^{1\overline{2}}+E_2\eta^{2\overline{2}},
\end{cases}
\end{equation}
with $B_1,G_1,D_1,D_2,D_3\in\q[i]\setminus \{0\}$.
Let $\omega=\frac{i}{2}(\eta^{1\overline{1}}+\eta^{2\overline{2}}+\eta^{3\overline{3}}+\eta^{4\overline{4}})$ be the fundamental form of the diagonal Hermitian metric $g$ on $\mathfrak{g}$. Then, $g$ is SKT if, and only if, it holds $\del\delbar\omega=0$, i.e., 
\begin{equation}\label{eq:SKT}
|B_1|^2+|G_1|^2+ |D_2|^2-2\mathfrak{Re}(D_1\overline{E}_2)=0.
\end{equation}
Assume that \eqref{eq:SKT} holds. We now construct a non trivial $ABC$-Massey product.

Consider the forms
$
\eta^{124\overline{1}}$ and $\eta^{\overline{2}}.
$
Note that, by structure equations, $d\eta^{124\overline{1}}=d\eta^{\overline{2}}=0$, hence
\[
[\eta^{124\overline{1}}]_{BC}\in H_{BC}^{3,1}(\mathfrak{g}), \qquad [\eta^{\overline{2}}]\in H_{BC}^{0,1}(\mathfrak{g})
\]
are well defined Bott-Chern cohomology classes. Moreover, the forms $\ast\eta^{124\overline{1}}=\eta^{3\overline{234}}$ and $\ast\eta^{\overline{2}}=-\eta^{1234\overline{134}}$ are $\del\delbar$-closed, thus the forms $\eta^{124\overline{1}}$ and $\eta^{\overline{2}}$ are Bott-Chern harmonic forms and they define non vanishing Bott-Chern cohomology classes, i.e.,
\[
[\eta^{124\overline{1}}]_{BC}\neq 0, \qquad [\eta^{\overline{2}}]_{BC}\neq 0.
\]
Now, $\eta^{124\overline{1}}\wedge\eta^{\overline{2}}=\del\delbar(-\frac{1}{|B_1|^2}\eta^{34\overline{3}})$, so that the $ABC$-Massey product
\[
\mathcal{P}=\Bigl\langle [\eta^{124\overline{1}}], [\eta^{\overline{2}}], [\eta^{\overline{2}}]
\Bigr\rangle_{ABC}\in\frac{H_A^{2,2}(\mathfrak{g})}{[\eta^{\overline{2}}]\cup H_A^{2,1}(\mathfrak{g})}
\]
is well defined and represented  by the Aeppli cohomology class
\[
\left[\frac{1}{|B_1|^2}\eta^{34\overline{23}}\right]_A\in H_A^{2,2}(\mathfrak{g}).
\]
Note that $\ast\eta^{34\overline{23}}=\eta^{12\overline{14}}$, which is $d$-closed, hence, as a Aeppli-cohomology class, $\left[\frac{1}{|B_1|^2}\eta^{34\overline{23}}\right]_A\neq 0$.
We claim that $\left[\frac{1}{|B_1|^2}\eta^{34\overline{23}}\right]_A$ does not belong to the ideal $[\eta^{\overline{2}}]\cup H_A^{2,1}(\mathfrak{g})$.

Assume, by contradiction, that this is the case, i.e., there exist constants $\lambda_j\in\C$ and forms $R\in\bigwedge^{1,2}\mathfrak{g}_\C^*$, $S\in\bigwedge^{2,1}\mathfrak{g}_\C^*$ such that
\[
\frac{1}{|B_1|^2}\eta^{34\overline{23}}=\sum_{j=1}^{h_A^{2,1}}\lambda_j\eta^{\overline{2}}\wedge\xi^{j}+\del R +\delbar S,
\]
where $\{\xi^j\}_{j=1}^{h_A^{2,1}}$ is a basis of $\mathcal{H}_A^{2,1}(\mathfrak{g})=(\Ker\del\delbar \cap \bigwedge^{2,1}\mathfrak{g}_\C^*)\cap (\Ker d^* \cap \bigwedge^{2,1}\mathfrak{g}_\C^*)$. Notice that
\begin{align*}
\Ker\del\delbar\cap\textstyle\bigwedge^{2,1}\mathfrak{g}_\C^*=\Span_{\C}&\langle\eta^{12\overline{1}}, \eta^{12\overline{2}}, \eta^{12\overline{3}}, \eta^{12\overline{4}},\eta^{13\overline{1}}, \eta^{13\overline{2}}, \eta^{13\overline{3}}, \eta^{13\overline{4}},
\eta^{14\overline{1}},\eta^{14\overline{2}},
\eta^{14\overline{3}},\\
&\eta^{14\overline{4}},\eta^{23\overline{1}},
\eta^{23\overline{2}},
\eta^{23\overline{3}},\eta^{23\overline{4}},\eta^{24\overline{1}}, 
\eta^{24\overline{2}}, 
\eta^{24\overline{3}},
\eta^{24\overline{4}},
\eta^{34\overline{1}},
\eta^{34\overline{2}}\rangle
\end{align*}
whereas
\begin{align*}
\Ker d^*\cap\textstyle \bigwedge^{2,1}\mathfrak{g}_\C^*=\displaystyle\langle \eta^{13\overline{3}},\eta^{13\overline{4}},\eta^{14\overline{1}}-\frac{\overline{D}_1}{G_1}\eta^{12\overline{4}}, \eta^{14\overline{2}}-\frac{\overline{B}_1}{G_1}\eta^{12\overline{3}},\eta^{14\overline{3}}, \eta^{14\overline{4}},\\
\eta^{23\overline{2}}+\frac{1}{\overline{D}_2}\left(\frac{E_2\overline{D}_1}{D_1}-\overline{E}_2\right)\eta^{23\overline{1}}-\frac{E_2}{D_1}\eta^{13\overline{1}},\eta^{23\overline{3}},\eta^{23\overline{4}},\\
\eta^{24\overline{1}}-\frac{B_1\overline{D}_1}{D_1\overline{D}_2}\eta^{23\overline{1}}+\frac{B_1}{D_1}\eta^{13\overline{1}}-\frac{\overline{D}_2}{G_1}\eta^{12\overline{4}},\\
\eta^{24\overline{2}}+\frac{\overline{E}_2}{G_1}\eta^{12\overline{4}}, \eta^{24\overline{3}},\eta^{24\overline{4}},\eta^{34\overline{1}},\eta^{34\overline{2}},\eta^{34\overline{3}},\eta^{34\overline{4}}
\rangle
\end{align*}
so that
\begin{align*}
\mathcal{H}_A^{2,1}(\mathfrak{g})=\langle \eta^{13\overline{3}},\eta^{13\overline{4}},\eta^{14\overline{1}}-\frac{\overline{D}_1}{G_1}\eta^{12\overline{4}}, \eta^{14\overline{2}}-\frac{\overline{B}_1}{G_1}\eta^{12\overline{3}},\eta^{14\overline{3}}, \eta^{14\overline{4}},\\
\eta^{23\overline{2}}+\frac{1}{\overline{D}_2}\left(\frac{E_2\overline{D}_1}{D_1}-\overline{E}_2\right)\eta^{23\overline{1}}-\frac{E_2}{D_1}\eta^{13\overline{1}},\eta^{23\overline{3}},\eta^{23\overline{4}},\\
\eta^{24\overline{1}}-\frac{B_1\overline{D}_1}{D_1\overline{D}_2}\eta^{23\overline{1}}+\frac{B_1}{D_1}\eta^{13\overline{1}}-\frac{\overline{D}_2}{G_1}\eta^{12\overline{4}},
\eta^{24\overline{2}}+\frac{\overline{E}_2}{G_1}\eta^{12\overline{4}},\\ \eta^{24\overline{3}},\eta^{24\overline{4}},\eta^{34\overline{1}},\eta^{34\overline{2}}
\rangle.
\end{align*}
Set $\xi^1:=\eta^{13\overline{3}}$, $\xi^2:=\eta^{13\overline{4}}$ and so on. We point out that $h_A^{2,1}=15.$

Then, by taking the $L^2$-norm of $\eta^{34\overline{23}}$, we obtain
\begin{align}
0\neq \frac{1}{|B_1|^2}\,||\eta^{34\overline{23}}||^2&=\frac{1}{|B_1|^2}\int_M\eta^{34\overline{23}}\wedge\ast\eta^{34\overline{23}}\nonumber\\
&=\int_M \sum_{j=1}^{15}\lambda_j\eta^{\overline{2}}\wedge\xi^{j}\wedge \eta^{12\overline{14}}+ \int_M \del R\wedge \eta^{12\overline{14}} +\int_M \delbar S\wedge \eta^{12\overline{14}}=0,\label{eq:ABC-stokes}
\end{align}
since $\xi^j\neq \eta^{34\overline{3}}$ for every $j\in\{1,\dots, 15\}$ and the first term of \eqref{eq:ABC-stokes} vanishes; by Stokes' theorem, being $d\eta^{12\overline{14}}=0$, also the second and third terms vanish. Therefore, we have a contradiction, i.e., the $ABC$-Massey product on $(\mathfrak{g},J)$
$$\mathcal{P}:=\Bigl\langle[\eta^{124\overline{1}}]_{BC}, [\eta^{\overline{2}}]_{BC}, [\eta^{\overline{2}}]_{BC}\Bigr\rangle_{ABC}\neq 0,$$
is not vanishing. Thus, by Lemma \ref{lemma:corrinvABC} the product $\mathcal{P}$ corresponds to non vanishing $ABC$-Massey product on $(M,J)$, which, consequently, is never geometrically-$BC$-formal.

As a result, each element of the family of manifolds \eqref{eq:struct_eq_SKT_non_geom_BC} satisfying assumption \eqref{eq:SKT} admits a SKT metric, i.e., the diagonal metric with respect to the coframe $\{\eta^1,\eta^2,\eta^3,\eta^4\}$, but it admits a non vanishing $ABC$-Massey product, so it never geometrically-$BC$-formal. \end{proof}

\section{K\"ahler solvmanifolds and Bott-Chern formality}\label{sec:5}

In this section, we prove that every K\"ahler solvmanifold is geometrically-$BC$-formal, hence answering a question in \cite{MS}: despite the existence of a manifold satisfying the $\del\delbar$-lemma and yet admitting a non vanishing $ABC$-Massey product \cite{SfeTom1}, it remains unclear whether $ABC$-Massey product provide an obstruction to stronger cohomological properties on a compact complex manifold. As a starting point, the authors in \cite{MS} suggest that one may look at the class of K\"ahler solvmanifolds, although they seem inclined to think that this class of manifolds might not be decisive. We will indeed prove that K\"ahler solvmanifolds are metrically formal, so that every $ABC$-Massey product vanishes.

We start by recalling the fundamental result of Hasegawa on the structure of K\"ahler solvmanifolds.
\begin{theorem}
A compact solvmanifold admits a K\"ahler structure if and only if it is a finite quotient of a complex torus which has a structure of a complex torus bundle over a complex torus. In particular, a compact solvmanifold of completely solvable type has a K\"ahler structure if and only
if it is a complex torus.
\end{theorem}
In \cite[Example 4]{Hase}, the author gives a characterization of the geometric structure of a K\"ahler non toric solvmanifold, which we now recall.

A K\"ahler solvmanifold $X$ arises as the compact quotient $\Gamma\backslash G$ of a simply connected solvable Lie group $G=\C^l\rtimes_{\phi}\R^{2k}$ defined by a map
\begin{gather*}
\phi\colon \R^{2k}\rightarrow \text{Aut}(\C^l)\\
\phi(t_iE_i)\cdot\, ^t(z_1,\dots,z_l):=\,^t\left(e^{\sqrt{-1}\eta^it_i}z_1,\dots, e^{\sqrt{-1}\eta^it_i}z_l\right)=e^{\sqrt{-1}\eta^it_i}\mathbb{I}\cdot \,^t(z_1,\dots, z_l),
\end{gather*}
where $E_i$ is the $i$-th unit vector of $\R^{2k}$ and $e^{\sqrt{-1}\eta^i}$ is the $s_i$-th root of unity, for $i=1,\dots, 2k$. The discrete subgroup $\Gamma:=\Z^{2l}\rtimes_{\phi} \Z^{2k}$ is defined so that $\phi(\Z^{2k})(\Z^{2l})\subset \Z^{2l}$ and, hence, $\Gamma$ a lattice of $G$.

Under the natural identification $\R^{2k}\simeq \C^{k}$, let $W=(z_{l+1},\dots,z_{l+k})\in\C^k$ be the coordinates of a point in $\C^k$ with respect to the standard basis of $\C^k$. Then, the map $\phi$ can be extended to
\begin{equation}\label{eq:semidirect_phi_KS}
\phi(W)\cdot \,^t(z_1,\dots,z_l)=e^{\sqrt{-1}(\sum_{i=1}^k\mathfrak{Re}(z_{l+i})\eta^{2l+2i-1}-\mathfrak{Im}(z_{l+i})\eta^{2l+2i})}\mathbb{I}\cdot\,^t(z_1, \dots,z_l).
\end{equation}
In particular, if $z=\,^t(z_1,\dots,z_l,z_{l+1},\dots, z_{l+k}),z'=\,^t(z_1',\dots,z_l',z_{l+1}',\dots, z_{l+k}')\in G$, then the product  rule $\ast$ of $G$ is defined as
\[
z'*z=\left(\phi(z_{l+1}',\dots,z_{l+k}')\cdot\, ^t(z_1,\dots,z_l)+\,^t(z_1',\dots,z_l'),\,^t(z_{l+i}+z_{l+1}',\dots, z_{l+k}+z_{l+k}')\right).
\]
The Lie algebra $\mathfrak{g}$ of $G$ is spanned by the set of $G$-left-invariant vector fields
\[
\left\{ X_1,\dots X_{2l},X_{2l+1}\dots, X_{2l+2k}\right\}
\]
which satisfy the bracket relations
\begin{equation}\label{eq:bracket_KS}
[X_{2l+2i},X_{2j-1}]=-X_{2j}, \qquad [X_{2l+2i},X_{2j}]=X_{2j-1},
\end{equation}
for $i=1,\dots, k$, $j=1, \dots, l$, whereas any other bracket vanishes.  The standard $G$-left-invariant almost complex structure $J$ on $G$ is defined by
\begin{equation}\label{eq:cplx-strct_KS}
JX_{2j-1}=X_{2j}, \quad JX_{2j}=-X_{2j-1}, \quad JX_{2l+2i-1}=X_{2l+2i}, \quad JX_{2l+2i}=-X_{2l+2i-1},
\end{equation}
for $i=1,\dots,k$, $j=1,\dots,k$, so that $\mathfrak{g}_{\C}$ decomposes as
\[
\mathfrak{g}_{\C}=\mathfrak{g}^{1,0}\oplus \mathfrak{g}^{0,1},
\]
where $\mathfrak{g}^{1,0}=\langle Z_1,\dots, Z_{l},Z_{l+1}, \dots, Z_{l+k}\rangle$ and $\mathfrak{g}^{0,1}=\langle \overline{Z_1},\dots, \overline{Z_{l}},\overline{Z_{l+1}}, \dots, \overline{Z_{l+k}}\rangle$ are, respectively, the $\pm i$-eigenspaces of $J$ and 
\[
Z_{j}:=\frac{1}{2}\left(X_{2j-1}-iX_{2j}\right), \quad Z_{l+i}:=\frac{1}{2}\left(X_{2l+2i-1}-iX_{2l+2i}\right)
\]
for $i=1,\dots,k$, $j=1,\dots, l$.

If we denote by $\displaystyle\left\{e^1,\dots, e^{2l},e^{2l+1},\dots, e^{2l+2k}\right\}$ the coframe of $G$-left-invariant forms on $G$ dual to $\{X_1,\dots,X_{2l},X_{2l+1},\dots, X_{2l+2k}\}$, dualizing the relations \eqref{eq:bracket_KS}, we obtain the structure equations
\begin{align*}
\begin{cases}
\,\, de^{2j-1} &=\displaystyle-\sum_{i=1}^ke^{2j}\wedge e^{2l+2i}, \quad j=1,\dots, l,\vspace{0.2cm}\\
\,\, de^{2j} &=\displaystyle\sum_{i=1}^ke^{2j-1}\wedge e^{2l+2i},  \quad j=1,\dots, l,\vspace{0.2cm}\\
\,\,de^{2l+2i-1} &=0, \quad i=1,\dots, k\vspace{0.2cm}\\
\,\,de^{2l+2i} &=0, \quad i=1,\dots, k.
\end{cases}
\end{align*}
The complex structure $J$ defined in \eqref{eq:cplx-strct_KS} determines the basis of $(\mathfrak{g}_{\C}^*)^{1,0}$ which we denote by $\left\{\varphi^1,\dots, \varphi^l,\varphi^{l+1},\dots, \varphi^{l+k}\right\}$, where
\begin{equation}\label{eq:complex_forms_KS}
\varphi^j:=e^{2j-1}+ie^{2j}, \qquad \varphi^{l+i}:=e^{2l+2i-1}+ie^{2l+2i},
	\end{equation}
for $j=1,\dots, l$, $i=1,\dots, k$, which satisfy the complex structure equations
\begin{align}\label{eq:complex_forms_struct_eq_KS}
\begin{cases}
\,\, d\varphi^j&=\displaystyle\frac{1}{2}\sum_{i=1}^k\varphi^{j}\wedge\varphi^{l+i}-\varphi^{j}\wedge\varphi^{\overline{l+i}}, \quad j=1,\dots, l\vspace{0.2cm}\\
\,\,d\varphi^{l+i}&=0, \quad i=1,\dots,k.
\end{cases}
\end{align}
As a result, the manifold $X=(\Gamma\backslash G,J)$ is a K\"ahler solvmanifold endowed with a $G$-invariant complex structure, since $J$ is $G$-left-invariant and hence descends to $\Gamma\backslash G$. An invariant K\"ahler metric on $X$ is, for example, the canonical diagonal  with respect to the coframe $\{\varphi^1, \dots, \varphi^{l+k}\}$, as we will prove in Corollary \ref{cor:inv_Kahler_metric_KS}.

We now give an explicit expression for the differential of any $G$-left-invariant $(p,q)$-form on $X$. At this level, we consider the elements of the natural basis
\begin{equation}\label{eq:basis_inv_forms_KS}
\left\{\varphi^{a_1}\wedge\dots \wedge\varphi^{a_p}\wedge\varphi^{\overline{b_1}}\wedge\dots\wedge\varphi^{\overline{b_q}}\colon 1\leq a_1<\dots <a_p\leq l+k, 1\leq b_1<\dots<b_q\leq l+k\right\}.
\end{equation}
\begin{notation*}
From now on, we will write the holomorphic and the anti-holomorphic components of any element of the basis \eqref{eq:basis_inv_forms_KS} by listing first the $(1,0)$-forms (respectively, $(0,1)$-forms) belonging to $\{\varphi^j\}_{j=1}^l$ (respectively, $\{\varphi^{\overline{j}}\}_{j=1}^l$) and then the $(1,0)$-forms (respectively, the $(0,1)$-forms) belonging to $\{\varphi^{l+i}\}_{i=1}^k$ (respectively, to $\{\varphi^{\overline{l+i}}\}_{i=1}^k$), i.e., we will adopt the following notation
\begin{equation}\label{eq:simpleform_KS}
\alpha=\varphi^{j_1\dots j_r}\wedge\varphi^{ l+i_1\dots l+i_{p-r}}\wedge\varphi^{\overline{m}_1\dots \overline{m}_s}\wedge\varphi^{\overline{l+n_{1}}\dots\overline{l+n_{q-s}}},
\end{equation}
with  $r\in\{\max(0,p-k),\dots, p\}$, $s\in\{\max(0,q-k),\dots, q\}$, $j_1,\dots,j_r\in\{1,\dots,l\}$, $i_1,\dots, i_{p-r}\in\{1,\dots, k\}$, $m_1,\dots,m_s\in\{1,\dots,l\}$, $n_1,\dots, n_{q-s}\in\{1,\dots, k\}$.
\end{notation*}
With respect to this notation the differential acts as in the following lemma.
\begin{lemma}\label{lemma:simpleform_KS}
For any invariant $(p,q)$-form $\alpha\in X$ written as in \eqref{eq:simpleform_KS}, it holds
\begin{align*}\label{eq:diff_simpleform_KS}
d\alpha=\frac{s-r}{2}\left(\sum_{i=1}^k\varphi^{l+i}\right)\wedge\alpha 
+\frac{r-s}{2}\left(\sum_{i=1}^k\varphi^{\overline{l+i}}\right)\wedge\alpha.
\end{align*}
In particular,
\[
\del\alpha=\frac{s-r}{2}\left(\sum_{i=1}^k\varphi^{l+i}\right)\wedge\alpha
\]
and
\[
\delbar\alpha=\frac{r-s}{2}\left(\sum_{i=1}^k\varphi^{\overline{l+i}}\right)\wedge\alpha.
\]
\end{lemma}

\begin{proof}
We see that, by Leibniz rule, we have
\begin{align*}
d\alpha=& \,d\left(\varphi^{j_1\dots j_r}\wedge\varphi^{ l+i_1\dots l+i_{p-r}}\wedge\varphi^{\overline{m}_1\dots \overline{m}_s}\wedge\varphi^{\overline{l+n_{1}}\dots\overline{l+n_{q-s}}}\right)\\
=& \,d\left(\varphi^{j_1\dots j_r}\wedge\varphi^{ l+i_1\dots l+i_{p-r}}\right)\wedge\varphi^{\overline{m}_1\dots \overline{m}_s}\wedge\varphi^{\overline{l+n_{1}}\dots\overline{l+n_{q-s}}}\\
&+(-1)^p\varphi^{j_1\dots j_r}\wedge\varphi^{ l+i_1\dots l+i_{p-r}}\wedge d\left(\varphi^{\overline{m}_1\dots \overline{m}_s}\wedge\varphi^{\overline{l+n_{1}}\dots\overline{l+n_{q-s}}}\right)\\
=& \,\Omega_1 + (-1)^p \Omega_2.
\end{align*}
Let us first focus on the term $\Omega_1=d\left(\varphi^{j_1\dots j_r}\wedge\varphi^{ l+i_1\dots l+i_{p-r}}\right)\wedge\varphi^{\overline{m}_1\dots \overline{m}_s}\wedge\varphi^{\overline{l+n_{1}}\dots\overline{l+n_{q-s}}}$. By the structure equations and graded-commutativity of the wedge product, it holds that
\begin{align*}
\Omega_1=& \,\frac{1}{2}\varphi^{j_1}\wedge\left(\sum_{i=1}^k\varphi^{l+i}-\sum_{i=1}^k\varphi^{\overline{l+i}}\right)\wedge\varphi^{j_2\dots j_r}\wedge\varphi^{ l+i_1\dots l+i_{p-r}}\wedge\varphi^{\overline{m}_1\dots \overline{m}_s}\wedge\varphi^{\overline{l+n_{1}}\dots\overline{l+n_{q-s}}}\\
&+\dots \\
&+(-1)^{r-1}\varphi^{j_1\dots j_{r-1}}\wedge\frac{1}{2}\varphi^{j_r}\wedge\left(\sum_{i=1}^k\varphi^{l+i}-\sum_{i=1}^k\varphi^{\overline{l+i}}\right)\wedge\varphi^{ l+i_1\dots l+i_{p-r}}\wedge\varphi^{\overline{m}_1\dots \overline{m}_s}\wedge\varphi^{\overline{l+n_{1}}\dots\overline{l+n_{q-s}}}\\
=& \, -\frac{r}{2}\left(\sum_{i=1}^k\varphi^{l+i}-\varphi^{\overline{l+i}}\right)\wedge\varphi^{j_1\dots j_r}\wedge\varphi^{ l+i_1\dots l+i_{p-r}}\wedge\varphi^{\overline{m}_1\dots \overline{m}_s}\wedge\varphi^{\overline{l+n_{1}}\dots\overline{l+n_{q-s}}}.
\end{align*}
Analogously, for other summand $\Omega_2=\varphi^{j_1\dots j_r}\wedge\varphi^{ l+i_1\dots l+i_{p-r}}\wedge d\left(\varphi^{\overline{m}_1\dots \overline{m}_s}\wedge\varphi^{\overline{l+n_{1}}\dots\overline{l+n_{q-s}}}\right)$ we obtain the explicit expression
\begin{align*}
\Omega_2=& \,\varphi^{j_1\dots j_r}\wedge\varphi^{ l+i_1\dots l+i_{p-r}}\wedge \frac{1}{2}\varphi^{\overline{m}_1}\wedge\left(\sum_{i=1}^k\varphi^{\overline{l+i}}-\sum_{i=1}^k\varphi^{l+i}\right)\wedge\varphi^{\overline{m}_2\dots \overline{m}_s}\wedge\varphi^{\overline{l+n_{1}}\dots\overline{l+n_{q-s}}}\\
&+ \dots\\
&+(-1)^s \varphi^{j_1\dots j_r}\wedge\varphi^{ l+i_1\dots l+i_{p-r}}\wedge \varphi^{\overline{m}_1\dots \overline{m}_{s-1}}\wedge\frac{1}{2}\varphi^{\overline{m_s}}\wedge\left(\sum_{i=1}^k\varphi^{\overline{l+i}}-\sum_{i=1}^k\varphi^{l+i}\right)\wedge\varphi^{\overline{l+n_{1}}\dots\overline{l+n_{q-s}}}\\
=& \, (-1)^p\frac{s}{2}\left(\sum_{i=1}^k\varphi^{l+i}-\varphi^{\overline{l+i}}\right)\wedge\varphi^{j_1\dots j_r}\wedge\varphi^{ l+i_1\dots l+i_{p-r}}\wedge\varphi^{\overline{m}_1\dots \overline{m}_s}\wedge\varphi^{\overline{l+n_{1}}\dots\overline{l+n_{q-s}}}.
\end{align*}
Hence, by adding $\Omega_1$ and $(-1)^p\Omega_2$, we obtain
\begin{align*}
d\alpha&=\frac{s-r}{2}\left(\sum_{i=1}^k\varphi^{l+i}\right)\wedge\varphi^{j_1\dots j_r}\wedge\varphi^{ l+i_1\dots l+i_{p-r}}\wedge\varphi^{\overline{m}_1\dots \overline{m}_s}\wedge\varphi^{\overline{l+n_{1}}\dots\overline{l+n_{q-s}}}\\
&+\frac{r-s}{2}\left(\sum_{i=1}^k\varphi^{\overline{l+i}}\right)\wedge\varphi^{j_1\dots j_r}\wedge\varphi^{ l+i_1\dots l+i_{p-r}}\wedge\varphi^{\overline{m}_1\dots \overline{m}_s}\wedge\varphi^{\overline{l+n_{1}}\dots\overline{l+n_{q-s}}},
\end{align*}
which yields the thesis.\\
\end{proof}
As an immediate application of Lemma \ref{lemma:simpleform_KS}, we observe the following.
\begin{cor}
Let $\alpha$ be any invariant $(p,q)$-form on $X$, written as in \eqref{eq:simpleform_KS}. If $r=s$, then $d\alpha=0$.
\end{cor}
We have then an explicit expression for a invariant K\"ahler metric on $X$.
\begin{cor}\label{cor:inv_Kahler_metric_KS}
The canonical Hermitian metric $g$ on $X$ with fundamental form $$\omega=\frac{i}{2}\sum_{j=1}^l\varphi^{j\overline{j}}+\frac{i}{2}\sum_{i=1}^k\varphi^{l+i\,\overline{l+i}}$$ is K\"ahler. 
\end{cor}
\begin{proof}
By structure equations \eqref{eq:cplx-strct_KS} and Lemma \ref{lemma:simpleform_KS}, we have that $d\varphi^{l+i\,\overline{l+i}}=0$, for every $i\in\{1,\dots,k\}$, and
\[
d\varphi^{j\overline{j}}=\frac{1-1}{2}\sum_{i=1}^k\varphi^{l+i}\wedge\varphi^{j\overline{j}}+\frac{1-1}{2}\sum_{i=1}^k\varphi^{\overline{l+i}}\wedge\varphi^{j\overline{j}}=0,
\]
for every $j\in\{1,\dots, l\}$. Hence, $d\omega=0$.\\
\end{proof}

We now adapt the techniques in \cite[Section 2.5]{AngKas} to compute the cohomology of $X$. We preliminarly observe the following.
\begin{lemma}
$X$ is a complex solvmanifold of splitting type, i.e., $X$ is a solvmanifold $X =\Gamma\backslash G$ endowed with a $G$-left-invariant
complex structure $J$ such that $G$ is the semidirect product $\C^n\rtimes_{\phi} N$ so that:
\begin{enumerate}[label=(\roman{*}), ref=(\roman{*})]
\item $N$ is a simply connected $2m$-dimensional nilpotent Lie group endowed with
an $N$-left-invariant complex structure $J_N$ ; (denote the Lie algebras of $\C^n$ and $N$ by $\mathfrak{a}$
and, respectively, $\mathfrak{n}$),
\item For any $t\in \C^n$, it holds that $\phi(t)\in\text{GL}(N)$ is a holomorphic automorphism of $N$ with
respect to $J_N$ ;
\item $\phi$ induces a semisimple action on $\mathfrak{n}$;
\item $G$ has a lattice $\Gamma$; (then $\Gamma$ can be written as $\Gamma=\Gamma_{\C^n} \rtimes_{\phi} \Gamma_N$ such that $\Gamma_{\C^n}$ and $\Gamma_{N}$ are lattices of $\C^n$ and, respectively, $N$, and, for any $t\in\Gamma_{\C^n}$, it holds $\phi(t)(\Gamma_N)\subseteq\Gamma_N)$;
\item The inclusion $\bigwedge^{\bullet,\bullet}\mathfrak{n}_\C^*\hookrightarrow \mathcal{A}^{\bullet,\bullet}(\Gamma_N\backslash N)$ induces the isomorphism
$$
H^{\bullet}(\textstyle\bigwedge^{\bullet,\bullet}\mathfrak{n}_\C^*,\delbar)\simeq H^{\bullet,\bullet}_{\delbar}(\Gamma_N\backslash N).
$$
\end{enumerate}
\end{lemma}
\begin{proof}
We recall that $X$ is a solvmanifold
$\Gamma\backslash G$ endowed with a $G$-invariant complex structure $J$. The universal cover of $X$ is the Lie group $G=\C^{l}\rtimes_{\phi} \C^{k}$ which admits the lattice $\Gamma=\Z^{l}\rtimes_{\phi} \Z^k$. Note that $\C^l$ is the nilradical of $G$ with Lie algebra $\mathfrak{n}$ and $\C^k$ is the factor with Lie algebra $\mathfrak{a}$, so that the Lie algebra of $G$ decomposes as
\[
\mathfrak{g}=\mathfrak{n}\oplus \mathfrak{a}.
\]
\emph{(i)} The factor $\C^l$ is a  simply connected abelian Lie group of real dimension $2l$, i.e., it is clearly a  simply connected nilpotent Lie group of real dimension $2l$. The standard complex structure $J_{\C^l}$ of $\C^l$ is determined by the standard frame $\{\frac{\del}{\del z_j}\}_{j=1}^l$ and its dual frame $\{dz_j\}_{j=1}^l$ and it is clearly invariant under the action of $\C^l$ on itself by traslation, i.e., $J_{\C^l}$ is $\C^l$-left-invariant.\\
\emph{(ii)} For any $W=(z_{l+1},\dots, z_{l+k})\in\C^k$, the map $\phi(W)$ with expression \eqref{eq:semidirect_phi_KS} is clearly a holomorphic automorphism of $\C^l$ with respect to $J_{\C^l}$.\\
\emph{(iii)} For every $W=(z_{l+1},\dots, z_{l+k})\in\C^k$ map $\phi$ induces a semisimple action on $\mathfrak{n}$ by
\begin{equation}\label{eq:induced_phi}
(d\phi(W))_e=e^{\sqrt{-1}(\sum_{i=1}^k\mathfrak{Re}(z_{l+i})\eta^{2l+2i-1}-\mathfrak{Im}(z_{l+i})\eta^{2l+2i})}\mathbb{I}\in \text{Gl}(\mathfrak{n}).
\end{equation}
\emph{(iv)} The universal cover $G$ admits a lattice $\Gamma$ by hypotesis.\\
\emph{(v)} Since $\Z^l\backslash \C^l$ is biholomorphic to a complex $l$-dimensional torus $\mathbb{T}^l$, the inclusion $\bigwedge^{\bullet,\bullet}\mathfrak{n}_\C^*\hookrightarrow \mathcal{A}^{\bullet,\bullet}(\mathbb{T}^l)$ induces an isomorphism in the Dolbeault cohomology, i.e.,
\[
H^{\bullet}(\textstyle\bigwedge^{\bullet,\bullet}\mathfrak{n}_\C^*,\delbar)=(\bigwedge^{\bullet,\bullet}\mathfrak{n}_\C^*,\delbar\equiv 0)\simeq H_{\delbar}^{\bullet,\bullet}(\mathbb{T}^l).
\]

\end{proof}
We are in the position to apply the results for complex solvmanifolds of splitting type as in \cite[Section 2.5]{AngKas}. With respect to the standard frame $\frac{\del}{\del z_j}$ of $\C^l$-left-invariant $(1,0)$-vector fields, the induced action \eqref{eq:induced_phi} of $\phi$ on $\mathfrak{n}^{1,0}$ is given by
\[
\C^l\ni W=(z_{l+1},\dots,z_{l+k})\mapsto \chi(W)\mathbb{I},
\]
with $\chi$ the character of $\C^k$ determined by $\chi(W)=e^{\sqrt{-1}(\sum_{i=1}^k\mathfrak{Re}(z_{l+i})\eta^{2l+2i-1}-\mathfrak{Im}(z_{l+i})\eta^{2l+2i})}$. Then, $\chi$ can be extended to a character of $G$ and the set  
\begin{equation}\label{AK_complex_forms_KS}
\{\chi^{-1}dz^{1}, \dots, \chi^{-1}dz^l,\sqrt{-1}(\eta^{2l+1}+\sqrt{-1}\eta^{2l+2})dz^{l+1},\dots,\sqrt{-1}(\eta^{2l+2k-1}+\sqrt{-1}\eta^{2l+2k})dz^{l+k}\}
\end{equation}
is a coframe of $G$-left-invariant $(1,0)$-forms on $G$. More precisely, the coframe \eqref{AK_complex_forms_KS} coincides with the coframe of $G$-left-invariant $(1,0)$-forms defined in \eqref{eq:complex_forms_KS} with structure equations \eqref{eq:complex_forms_struct_eq_KS}, i.e.,
\begin{equation*}
\begin{cases}
\varphi^j=\chi^{-1}dz^j, \qquad j\in\{1,\dots, k\}\\
\varphi^{l+i}=\sqrt{-1}(\eta^{2l+2i-1}+\sqrt{-1}\eta^{2l+2i})dz^{l+i}, \qquad i\in\{1,\dots, k\}.
\end{cases}
\end{equation*}

By Lemma \cite[Lemma 2.12]{AngKas} (see also the proof of \cite[Lemma 2.2]{Kas}), the unique unitary characters $\beta,\gamma\in \text{Hom}(\C^k,\C^*)$ such that $\chi\cdot \beta^{-1}$ and $\overline{\chi}\cdot \gamma^{-1}$ are holomorphic characters are exactly
\[
\beta=\chi, \qquad \gamma=\overline{\chi}=\chi^{-1},
\]
since both $\chi$ and $\overline{\chi}$ are unitary and holomorphic characters. For any multi-index $J=(j_1, \dots, j_p)\subset \{1,\dots, l\}$ of length $|J|=p$ and set of characters $\{\delta_1,\dots,\delta_l\}$, we set $\delta_J:=\delta_{j_1}\cdot \dots \cdot  \delta_{j_p}$.

In the adapted notation, the bi-complex $B_{\Gamma}^{\bullet,\bullet}$ \cite[Theorem 2.13]{AngKas} which computes the Dolbeault cohomology is
\begin{align}
B_{\Gamma}^{p,q}=\C\left\langle\varphi^{J}\wedge\varphi^{l+I}\wedge\varphi^{\overline{M}}\wedge\varphi^{\overline{l+N}}:\,|J|+|I|=p,\, |M|+|N|=q, \,\, \chi_J\cdot \chi_M^{-1}\restrict{\Gamma}\equiv 1 \right\rangle\subset \textstyle\bigwedge^{p,q}\mathfrak{g}_\C^*.
\end{align}
Note that $\chi_J=\chi^{|J|}$ and $\chi_M^{-1}=\chi^{-|M|}$, so the condition
\[
\chi_J\cdot \chi_M^{-1}\restrict{\Gamma}\equiv 1,
\]
reduces to
\[
\chi^{|J|-|M|}\restrict{\Gamma}\equiv 1,
\]
which is satisfied if and only if $|J|=|M|$.
\begin{rem}
Note that the defining property of the complex $B_{\Gamma}^{\bullet,\bullet}$ actually does not depend on the choice of the lattice $\Gamma$, therefore we will omit the dependency on $\Gamma$ from the notation and we will write only $B^{\bullet,\bullet}$. 
\end{rem}
The complex $B^{\bullet,\bullet}$ is given by
\[
B^{p,q}=\C\left\langle\varphi^{J}\wedge\varphi^{l+I}\wedge\varphi^{\overline{M}}\wedge\varphi^{\overline{l+N}}:\,|J|+|I|=p,\, |M|+|N|=q,\,\, |J|=|M|\right\rangle.
\]
Note that the invariant diagonal K\"ahler metric $g$ on $X$ defines, by restriction, a metric for the complex $B^{\bullet,\bullet}$, which we still denote by $g$.
\begin{lemma}\label{lemma:complex_B_Gamma}
The exterior differential on the complex $B^{\bullet,\bullet}$ is identically zero, i.e, $d\restrict{B^{\bullet,\bullet}}\equiv 0$. Moreover, the complex $B^{\bullet,\bullet}$ is closed under conjugation, i.e., $\overline{B^{p,q}}=B^{q,p}$, and it is closed with respect to $\ast_g$ and has a structure of an algebra with respect to $\wedge$
\end{lemma}
\begin{proof}
Let $\alpha$ be an element of the natural basis $B^{p,q}$, i.e.,
\begin{equation}\label{eq:simple_form_B_gamma_KS}
\alpha=\varphi^{j_1\dots j_{r}}\wedge\varphi^{l+i_1\dots l+i_{p-r}}\wedge{\varphi}^{\overline{m}_1\dots \overline{m}_{r}}\wedge\varphi^{\overline{l+n_1}\dots \overline{l+n_{q-r}}},
\end{equation}
with respect to the usual decomposition \eqref{eq:simpleform_KS}. By Lemma \ref{lemma:simpleform_KS}, it is immediate that $d\alpha=0$, hence, by linearity of the exterior differential, we obtain $d\restrict{B^{\bullet,\bullet}}\equiv 0$. 

Note that for any $\alpha$ of the form \eqref{eq:simple_form_B_gamma_KS}, the complex conjugation acts as
$$
\overline{\alpha}=\varphi^{\overline{j}_1\dots \overline{j}_{r}}\wedge\varphi^{\overline{l+i_1}\dots \overline{l+i_{p-r}}}\wedge{\varphi}^{m_1\dots m_{r}}\wedge\varphi^{l+n_1\dots l+n_{q-r}}.
$$
Hence $\overline{\alpha}\in B^{q,p}$ and $\overline{B^{p,q}}=B^{q,p}$.

If $g$ is the invariant diagonal K\"ahler metric for $B^{\bullet,\bullet}$, the coframe $\{\varphi^1,\dots,\varphi^{l+k}\}$ is orthonormal with respect to $g$. Let us denote by $\ast:=\ast_g$ the $\C$-antilinear Hodge operator
\[
\ast\colon \textstyle\bigwedge^{p,q}\mathfrak{g}_{\C}^*\rightarrow\bigwedge^{l+k-p,l+k-q}\mathfrak{g}_\C^*.
\] 

In particular, if $\{j_{r+1},\dots,j_{l}\}$ and $\{m_{r+1},\dots,m_{l}\}$ are the complementary sets in $\{1,\dots, l\}$ of $\{j_1,\dots, j_r\}$ and, respectively, $\{m_1,\dots, m_{r}\}$, and $\{i_{p-r+1},\dots, i_{k}\}$ and $\{n_{q-r+1},\dots, n_k\}$ are the complementary sets in $\{1,\dots,k\}$ of $\{i_{1},\dots, i_{k}\}$ and, respectively, of $\{n_1,\dots, n_{q-r}\}$, then for any element of the natural basis of $B^{\bullet,\bullet}$ with expression as in \eqref{eq:simple_form_B_gamma_KS}, up to a sign, the Hodge $\ast$ operator operates as
\[
\ast\alpha=\varphi^{j_{r+1}\dots j_{l}}\wedge\varphi^{l+i_{p-r+1}\dots l+i_{k}}\wedge{\varphi}^{\overline{m}_{r+1}\dots \overline{m}_{l}}\wedge\varphi^{\overline{l+n_{q-r+1}}\dots \overline{l+n_{k}}},
\]
i.e., for every $\alpha\in B^{\bullet,\bullet}$, also $\ast\alpha\in B^{\bullet,\bullet}$. Moreover, taking any two elements $\alpha$ and $\beta$ of the natural basis of $B^{\bullet,\bullet}$, i.e.,
\[
\alpha=\varphi^{J}\wedge\varphi^{l+I}\wedge\varphi^{\overline{M}}\wedge\varphi^{\overline{l+N}},\qquad \beta=\varphi^{J'}\wedge\varphi^{l+I'}\wedge{\varphi}^{\overline{M}'}\wedge\varphi^{\overline{l+N'}}
\]
with $|J|+|I|=p, |M|+|N|=q$, and $|J|=|M|$, and $|J'|+|I'|=r, |M'|+|N'|=s$, and $|J'|=|M'|$, then the expression of the wedge product of $\alpha$ and $\beta$ is given by
\[
\alpha\wedge \beta=(-1)^{|J'|(|M|+|I|)+|I|'(|M|+|N|)}\varphi^{J}\wedge\varphi^{J'}\wedge\varphi^{l+I}\wedge\varphi^{l+I'}\wedge\varphi^{\overline{M}}\wedge\varphi^{\overline{M}'}\wedge\varphi^{\overline{l+N}}\wedge\varphi^{\overline{l+N'}}.
\]
Since $|J|+|J'|=|M|+|M|'$, the form $\alpha\wedge\beta\in B^{\bullet,\bullet}$, i.e, $(B^{\bullet,\bullet},\wedge)$ is an algebra.\\
\end{proof}
We are now finally ready to describe explicitly the Dolbeault cohomology of any K\"ahler solvmanifold and, hence, by the Hodge decomposition and the K\"ahler identities, also its de Rham cohomology, the Bott-Chern cohomology and the Aeppli cohomology.
\begin{theorem}\label{thm:cohom_KS}
Let $X=(\Gamma\backslash G,J)$ be a K\"ahler solvmanifold with $G=\C^l\rtimes\C^k$. Let also $\{\varphi^1,\dots, \varphi^l,\varphi^{l+1},\dots, \varphi^{l+k}\}$ be the coframe of $G$-invariant forms   on $X$ determined by the structure equations
\[
\begin{cases}
\displaystyle d\varphi^j=\frac{1}{2}\varphi^j\wedge\left(\sum_{i=1}^k\varphi^{l+i}-\varphi^{\overline{l+i}}\right), \qquad j=1,\dots,l,\\
\displaystyle d\varphi^{l+i}=0, \qquad i=1, \dots, k.
\end{cases}
\]
Then, the Dolbeault cohomology spaces of $X$ are
\begin{equation*}
H_{\delbar}^{p,q}(X)\cong\C\left\langle\left[\varphi^{j_1\dots j_r}\wedge\varphi^{l+i_1\dots l+i_{p-r}}\wedge\varphi^{\overline{m}_1\dots \overline{m}_r}\wedge\varphi^{\overline{l+n_1}\dots \overline{l+n_{q-r}}}\right]\right\rangle_{\max(0,p-k,q-k)\leq r\leq \min(p,q)}
\end{equation*}
and the de Rham cohomology spaces of $X$ 
\[
H_{dR}^{s}(X;\C)\cong \bigoplus_{p+q=s}\C\left\langle\left[\varphi^{j_1\dots j_r}\wedge\varphi^{l+i_1\dots l+i_{p-r}}\wedge\varphi^{\overline{m}_1\dots \overline{m}_r}\wedge\varphi^{\overline{l+n_1}\dots \overline{l+n_{q-r}}}\right]\right\rangle_{\max(0,p-k,q-k)\leq r\leq \min(p,q)}.
\]

In particular, the Hodge numbers of $X$ are
\begin{equation}\label{eq:hodge_numbers_KS}
h_{\delbar}^{p,q}=\sum_{r=\max(0,p-k,q-k)}^{\min(p,q)}\begin{pmatrix}
l\\
r
\end{pmatrix}\left[\begin{pmatrix}
k\\
p-r
\end{pmatrix}+\begin{pmatrix}
k\\
q-r
\end{pmatrix}\right].
\end{equation}
\end{theorem}

\begin{proof}
Let us start by fixing $g$ the invariant diagonal  K\"ahler metric on $B^{\bullet,\bullet}$. Then, by Lemma \ref{lemma:complex_B_Gamma}, if $\alpha\in B^{\bullet,\bullet}$, then $\delbar\alpha=0$ and $\delbar\ast\alpha=0$, hence
\begin{equation}\label{eq:thm_coom_KS_1}
B^{\bullet,\bullet}=\mathcal{H}_{\delbar}^{\bullet,\bullet}(B^{\bullet,\bullet})\cong H^{\bullet,\bullet}(B^{\bullet,\bullet},\delbar),
\end{equation}
where the isomorphism is given by the projection onto the Dolbeault cohomology of the complex $B^{\bullet,\bullet}$. By \cite[Theorem 2.14]{AngKas}, the inclusion $B^{\bullet,\bullet}\hookrightarrow \mathcal{A}^{\bullet,\bullet}X$ induces the isomorphism in Dolbeault cohomology
\begin{equation}\label{eq:thm_coom_KS_2}
H^{\bullet,\bullet}(B^{\bullet,\bullet},\delbar)\cong H_{\delbar}^{\bullet,\bullet}(X),
\end{equation}
so we obtain that each for every $p,q$, for the Dolbeault cohomology spaces $H_{\delbar}^{p,q}(X)$ of $X$ satisfy
\begin{equation}\label{eq:dolb_cohomology_KS}
H_{\delbar}^{p,q}(X)\cong\C\left\langle\left[\varphi^{j_1\dots j_r}\wedge\varphi^{l+i_1\dots l+i_{p-r}}\wedge\varphi^{\overline{m}_1\dots \overline{m}_r}\wedge\varphi^{\overline{l+n_1}\dots \overline{l+n_{q-r}}}\right]\right\rangle_{\max(0,p-k,q-k)\leq r\leq \min(p,q)} ,
\end{equation}
and, by the Hodge decomposition, for every $k$, the de Rham cohomology of $X$ is
\begin{equation}\label{eq:deRham_cohomology_KS}
H_{dR}^{k}(X;\C)=\bigoplus_{p+q=k}\left\langle\left[\varphi^{j_1\dots j_r}\wedge\varphi^{l+i_1\dots l+i_{p-r}}\wedge\varphi^{\overline{m}_1\dots \overline{m}_r}\wedge\varphi^{\overline{l+n_1}\dots \overline{l+n_{q-r}}}\right]\right\rangle_{\max(0,p-k,q-k)\leq r\leq \min(p,q)}.
\end{equation}
As a straightforward consequence, the explicit description \eqref{eq:dolb_cohomology_KS} shows that the Hodge numbers of $X$ are given by
\begin{equation*}
h_{\delbar}^{p,q}=\sum_{r=\max(0,p-k,q-k)}^{\min(p,q)}\begin{pmatrix}
l\\
r
\end{pmatrix}\begin{pmatrix}
k\\
p-r
\end{pmatrix}\cdot \sum_{r=\max(0,p-k,q-k)}^{\max(p,q)}\begin{pmatrix}
l\\
r
\end{pmatrix}\begin{pmatrix}
k\\
q-r
\end{pmatrix},
\end{equation*}
which concludes the proof.\\
\end{proof}
Since $X$ is compact K\"ahler manifold,  $X$ satisfies the $\del\delbar$-lemma. Therefore, we have the following isomorphisms between the complex cohomologies of $X$, i.e.,
\[H_{\delbar}^{\bullet,\bullet}(M)\cong H_{BC}^{\bullet,\bullet}(X)\cong H_A^{\bullet,\bullet}(X).
\]
Thus, from Theorem \ref{thm:cohom_KS} we immediately obtain the following.
\begin{cor}
\[
H_{BC}^{p,q}(X)=H_A^{p,q}(X)\cong\C\left\langle\left[\varphi^{j_1\dots j_r}\wedge\varphi^{l+i_1\dots l+i_{p-r}}\wedge\varphi^{\overline{m}_1\dots \overline{m}_r}\wedge\varphi^{\overline{l+n_1}\dots \overline{l+n_{q-r}}}\right]\right\rangle_{\max(0,p-k,q-k)\leq r\leq \min(p,q)}.
\]
\end{cor}
Moreover, by the K\"ahler identities, it turns out that the spaces of harmonic forms with respect to the complex Laplacians coincide, i.e.,
\[
\mathcal{H}_{\Delta}^{\bullet,\bullet}(X)=\mathcal{H}_{\Delta_{\delbar}}^{\bullet,\bullet}(X)=\mathcal{H}_{\Delta_ {BC}}^{\bullet,\bullet}(X)=\mathcal{H}_{\Delta_{A}}^{\bullet,\bullet}(X).
\]
Therefore, the following result holds for any type of geometric formality (Kotschick, Dolbeault, Bott-Chern, and $ABC$).
\begin{theorem}\label{thm:formality_KS}
Let $X$ be any K\"ahler solvmanifold. Then $X$ is geometrically formal.
\end{theorem}
\begin{proof} 
Since every map of \eqref{eq:thm_coom_KS_1} and \eqref{eq:thm_coom_KS_2} is induced by either an inclusion or a projection in cohomology, we obtain that the space of harmonic forms of $X$ coincides with $B^{\bullet,\bullet}$, i.e.,
\[
\mathcal{H}_{\square}^{\bullet,\bullet}(X)=B^{\bullet,\bullet}, \qquad \square\in\{\Delta,\Delta_{\delbar},\Delta_{BC},\Delta_A\}.
\]
By Lemma \ref{lemma:complex_B_Gamma}, the pair $(B^{\bullet,\bullet},\wedge)$ is an algebra, so that also the spaces
$$
(\mathcal{H}_{\square}^{\bullet,\bullet}(X),\wedge), \qquad \square\in\{\Delta, \Delta_{\delbar},\Delta_{BC},\Delta_A\}
$$
are an algebra, which proves that $X$ is geometrically formal.
\end{proof}
As a direct consequence of Theorem \ref{thm:formality_KS}, we are able to prove that $ABC$-Massey product are an obstruction for the existence of a K\"ahler metric on solvmanifolds.
\begin{cor}\label{cor:ABC_prod_KS}
Let $X$ be any K\"ahler solvmanifold. Then every $ABC$-Massey product vanishes.
\end{cor}
\begin{proof}
By Theorem \ref{thm:formality_KS}  and \cite[Remark 4.7]{MS}, $X$ is also weakly formal in the sense of \cite{MS}. Since $X$ satisfies the $\del\delbar$-lemma, then $X$ is also strongly formal. But by \cite[Proposition 4.4]{MS}, we can conclude that every $ABC$-Massey product vanishes.
\end{proof}

\section{Quadruple ABC-Massey product}\label{sec:6}
In this section, we explicitly construct two families of solvmanifolds admitting a non vanishing quadruple $ABC$-Massey product.

Let $G':=(\C^4,\ast')$ be the complex Lie group with operation $\ast':\C^4\times \C^4\rightarrow\C^4$ defined by
\begin{equation*}
^t(y_1,y_2,y_3,y_4)\ast'\, ^t(z_1,z_2,z_3,z_4):=\,^t(z_1+y_1,e^{-y_1}z_2+y_2,e^{y_1}z_3+y_3,z_4+w_4+e^{y_1}y_2z_3).
\end{equation*}
More precisely, $G'$ is the element $6$ of the characterization of $4$-dimensional complex Lie groups by Nakamura in \cite[Section 6]{Nak}. Let now $G=(\C^4,\ast)$ be the complex Lie group endowed with the operation $\ast\colon \C^4\times\C^4\rightarrow\C^4$ defined by
\[
^t(y_1,y_2,y_3,y_4)\ast\, ^t(z_1,z_2,z_3,z_4):=\,^t(z_1+y_1,e^{-y_1}z_2+y_2,e^{y_1}z_3+y_3,z_4+y_4+\frac{1}{2}e^{y_1}y_2z_3-\frac{1}{2}e^{-y_1}y_3z_2).
\]
Then, one can see that $G$ and $G'$ are isomorphic as complex Lie groups via the invertible Lie group homomorphism $F\colon G'\rightarrow G$ given by
\begin{equation}\label{eq:map_F_quad}
F(z_1,z_2,z_3,z_4):=(z_1,z_2,z_3,z_4-\frac{1}{2}z_2z_3).
\end{equation}
From now on, we choose to use the presentation $G=(\C^4,\ast)$. Note that the operation $\ast$ of $G$ can be written as
\begin{align}\label{eq:multiplication_quad}
& ^t(y_1,y_2, y_3,y_4) \ast\, ^t(z_1,z_2,z_3,z_4):=\\
&\,^t\left(z_1+y_1,\,^t\left(\begin{pmatrix}
e^{-y_1} & 0\\
0 & e^{y_1}
\end{pmatrix}\begin{pmatrix}
z_2\\z_3 
\end{pmatrix}+ \begin{pmatrix}
y_2\\ y_3
\end{pmatrix}\right),y_4+\frac{1}{2} \,\, (y_2,y_3)\begin{pmatrix}
0 & 1\\ -1 & 0
\end{pmatrix}\begin{pmatrix}
e^{-y_1} & 0\\
0 & e^{y_1}
\end{pmatrix}\begin{pmatrix}
z_2 \\ z_3
\end{pmatrix}+z_4\right)\nonumber.
\end{align}
It is straightforward to check that a coframe of left-invariant holomorphic forms on $G$ is given in the standard coordinates $\{z_1,z_2,z_3,z_4\}$ of $\C^4$ by
\begin{equation}\label{eq:struct_eq_quad}
\varphi^1=dz_1, \quad \varphi^2=e^{z_1}dz_2, \quad \varphi^3=e^{-z_1}dz_3, \quad \varphi^4=dz_4-\frac{1}{2}z_2dz_3-\frac{1}{2}z_3dz_2,
\end{equation}
and they satisfy the structure equations
\begin{equation*}
d\varphi^1=0,\qquad 
d\varphi^2=\varphi^{12},\qquad
d\varphi^3=-\varphi^{13},\qquad
d\varphi^4=-\varphi^{23}.
\end{equation*}
Note that the dual $\mathfrak{g}_{\C}^*$ of the complexified Lie algebra $\mathfrak{g}_{\C}$ of $G$ decomposes as
\[
\mathfrak{g}_{\C}^*=\mathfrak{g}_{+}^{*}\oplus\mathfrak{g}_-^*, 
\]
where we denote by $\mathfrak{g}_+^*=\langle\varphi^1, \varphi^2, \varphi^3, \varphi^4\rangle$ the subspace of holomorphic left-invariant $1$-forms on $G$ and by $\mathfrak{g}_-^*=\langle \varphi^{\overline{1}},\varphi^{\overline{2}},\varphi^{\overline{3}},\varphi^{\overline{4}},\rangle$ the subspace of anti-holomorphic left-invariant $1$-forms on $G$.

The frame $\{Z_1,Z_2,Z_3,Z_4\}$ of $\mathfrak{g}_+$ dual to \eqref{eq:struct_eq_quad} is then given in holomorphic coordinates by
\[
Z_1=\frac{\del}{\del z_1}, \quad Z_2=e^{-z_1}\frac{\del}{\del z_2}+e^{-z_1}z_3\frac{\del}{\del z_4}, \quad Z_3=e^{z_1}\frac{\del}{\del z_3}+ e^{z_1}z_2\frac{\del}{\del z_4}, \quad Z_4=\frac{\del}{\del z_4},
\]
and by dualizing \eqref{eq:struct_eq_quad}, the only non vanishing brackets are
\begin{equation}\label{eq:bracket_quad}
[Z_1,Z_2]=-Z_2, \quad [Z_1,Z_3]=Z_3, \quad [Z_2,Z_3]=Z_4.
\end{equation}
By complex conjugation, we obtain the frame $\{\overline{Z}_1,\overline{Z}_2,\overline{Z}_3,\overline{Z}_4\}$ for $\mathfrak{g}_-$ and their brackets, so to obtain the analogous decomposition
\[
\mathfrak{g}_{\C}=\mathfrak{g}_+\oplus \mathfrak{g}_-.
\]
We recover the dual $\mathfrak{g}^*=\langle e^1,\dots, e^8\rangle_\R$ of the underying real Lie algebra of $G$ by defining the complex structure $J$ as $Je^{2j}:=e^{2j-1},$ for $j\in\{1,\dots, 4\}$, so that $$\varphi^j=:e^{2j-1}+ie^{2j},$$ for $j\in\{1,\dots, 4\}$. In particular, the real structure  equations of $\mathfrak{g}$ are
\begin{align*}
de^1=0, && de^2=0, && de^3=e^{13}-e^{24}, && de^4=e^{14}+e^{23},\\
de^5=-e^{15}+e^{26}, && de^6=-e^{16}-e^{25}, && de^7=-e^{35}+e^{46}, && de^8=-e^{36}-e^{45},\nonumber
\end{align*}
and the real Lie algebra $\mathfrak{g}$ of $G$ is then spanned by the left-invariant vector fields $e_1,\dots, e_8$ on $G$ which satisfy the bracket relations
\begin{align*}
&[e_1,e_3]=-e_3,\quad [e_1,e_4]=-e_4,\quad [e_1,e_5]=e_5,\quad [e_1,e_6]=e_6,\\
&[e_2,e_3]=-e_4,\quad [e_2,e_4]=e_3,\quad
[e_2,e_5]=e_6,\quad [e_2,e_6]=-e_5,\\
&[e_3,e_5]=e_7,\quad [e_3,e_6]=e_8,\quad [e_4,e_5]=e_8,\quad [e_4,e_6]=-e_7.
\end{align*}
It follows immediately that $\mathfrak{g}^1=[\mathfrak{g},\mathfrak{g}]=\langle e_3,e_4,e_5,e_6,e_7,e_8\rangle$ and $\mathfrak{g}^{2}=[\mathfrak{g}^1,\mathfrak{g}^1]=\langle e^7,e^8\rangle$ and $\mathfrak{g}^3=0$. Hence, $G$ is a $3$-step solvable non nilpotent complex Lie group.

We point out that $G$ has a structure of semidirect product, i.e., $G$ can be presented  as
\[
G\cong\C\ltimes_{\phi} N,
\]
where the factor $\C$ is simply-connected abelian (hence, nilpotent) and corresponds to the subalgebra $\mathfrak{c}=\langle Z_1\rangle\subset \mathfrak{g}_+^*$, whereas $N=\{(0,z_2,z_3,z_4)\}\leq G$ is the nilradical of $G$ and corresponds to the ideal $\mathfrak{n}=\langle Z_2,Z_3,Z_4\rangle\subset \mathfrak{g}_+^*$. Note that, via the isomorphism \eqref{eq:map_F_quad},  $N$ is isomorphic to the $3$-dimensional complex Heisenberg group
i.e., $N\cong \mathbb{H}(3;\C)$.

The semidirect product map $\phi\,\colon \C\rightarrow \text{Aut}(N)$ is defined, according to \eqref{eq:multiplication_quad}, by
\[
\phi(y_1)\, ^t(z_2,z_3,z_4)=\, ^t(e^{-y_1}y_2,e^{y_1}z_3,z_4),
\]
for every $z_1\in\C$, $^t(z_2,z_3,z_4)\in N$, where we have identified $N$ with $(\C^3,\star)$ and $$^t(y_2,y_3,y_4)\star\,^t(z_2,z_3,z_4):=\,^t(z_2+y_2,z_3+y_3, z_4+y_4+\frac{1}{2}y_2z_3-\frac{1}{2}z_2y_3)$$ for every $^t(y_2,y_3,y_4),\,^t(z_2,z_3,z_4)\in\C^3$.

We now construct a lattice for the Lie group $G$. Let us fix $A\in SL(2;\Z)$ and $e^{-\lambda},e^{\lambda}\in \R$ its real eigenvalues, with $\lambda>0$. In particular, there exists a matrix $P\in GL(2;\R)$ such that
\begin{equation}\label{eq:comm_int_matrix_quad}
P\cdot A \cdot P^{-1}=\begin{pmatrix}
e^{-\lambda} & 0\\
0 & e^{\lambda}
\end{pmatrix}=:\Lambda.
\end{equation}
Set $|P|:=\det(P)$. Along the lines of \cite[Example 3.4]{Yam}, we define the discrete subset $\Gamma$ of $\C^4$ as
\begin{equation*}
\Gamma_{\mu}:=\left\{\begin{pmatrix}\lambda a+i\mu b\\P(m_1+im_2)\\\frac{1}{2}|P|(h+ik)\end{pmatrix}: a,b,h,k\in \Z, m_1,m_2\in \Z^2\right\}, \qquad \mu\in \R.
\end{equation*}
\begin{lemma}\label{lemma:lattices_quad}
For $\mu=\pi$ and $\mu=\frac{\pi}{2}$, 
$\Gamma_{\mu}$ is a lattice of $G$.
\end{lemma}
\begin{proof}
First of all we prove that $\Gamma_{\pi}$ is a subgroup of $G$. 

Clearly the identity $e=\,^t(0,0,0,0)$ of $G$ belongs to $\Gamma$ by choosing $a=b=h=k=0$ and $m_1=m_2=\begin{pmatrix}
0\\0
\end{pmatrix}.$\\
Taking $\gamma,\gamma'\in\Gamma$, then by using the formula \eqref{eq:multiplication_quad} for $\ast$, we obtain
\begin{align*}
\gamma\ast \gamma'&=\begin{pmatrix}\lambda a+i2\pi b\\P(m_1+im_2)\\\frac{1}{2}|P|(h+ik)\end{pmatrix}\ast \begin{pmatrix}\lambda a'+i2\pi b'\\P(m_1'+im_2')\\\frac{1}{2}|P|(h'+ik')\end{pmatrix}\\
&=\begin{pmatrix}\lambda(a+a')+i2\pi (b+b')\\\begin{pmatrix}e^{-\lambda a -2\pi i b} & 0\\
0 & e^{\lambda a+i2\pi b}\end{pmatrix}
P(m_1'+im_2')+P(m_1+im_2)\\
\frac{1}{2}|P|(h+ik)+\frac{1}{2}\,^t(P(m_1+im_2))\begin{pmatrix}
0 & 1 \\ -1 & 0
\end{pmatrix}
\begin{pmatrix}
e^{-\lambda a -2\pi i b} & 0\\
0 & e^{\lambda a+i2\pi b}
\end{pmatrix}P(m_1'+im_2')+ \frac{1}{2}|P|(h'+ik')
\end{pmatrix}.
\end{align*}
Since $e^{i2\pi b}=1$ for every $b\in\Z$, and by \eqref{eq:comm_int_matrix_quad},
\[
\begin{pmatrix}
e^{-\lambda a-i2\pi b} & 0\\
0 & e^{\lambda a+i2\pi b}
\end{pmatrix}\cdot P=\begin{pmatrix}
(e^{-\lambda})^a & 0\\
0 & (e^{\lambda})^{a}
\end{pmatrix}\cdot P=\Lambda^a\cdot P=P\cdot A^a,
\]
then
\[
\begin{pmatrix}
e^{-\lambda a-i2\pi b} & 0\\
0 & e^{\lambda a+i2\pi b}
\end{pmatrix}P(m_1+im_2)+P(m_1'+im_2')=PA^a(m_1+im_2)+P(m_1'+im_2')\in P(\Z+i\Z)
\]
and
\begin{align*}
\frac{1}{2}\, ^t(P(m_1+im_2))\begin{pmatrix}
0 & 1 \\ -1 & 0
\end{pmatrix} 
\Lambda^a P(m_1'+im_2')&=\frac{1}{2}\, ^t(m_1+im_2) \,^tP\begin{pmatrix}
0 & 1 \\ -1 & 0
\end{pmatrix} P A^a(m_1'+im_2')\\
& =\frac{1}{2}\, ^t(m_1+im_2)\begin{pmatrix}
0 & |P|\\
-|P| & 0
\end{pmatrix}A^a(m_1'+im_2')\in \frac{1}{2}|P|(\Z+i\Z).
\end{align*}
Hence, $\gamma\ast \gamma'\in \Gamma$.\\
Finally, for any given $z=\,^t(z_1,z_2,z_3,z_4)\in G$, the inverse with respect to $\ast$ is
$$
z^{-1}=\,^t(-z_1,-\,^t\left(\begin{pmatrix}e^{z_1} & 0\\
0 & e^{-z_1}
\end{pmatrix}
\begin{pmatrix}z_2\\z_3
\end{pmatrix}\right),-z_4).
$$
Then, if $\gamma=\begin{pmatrix} \lambda a+i2\pi b\\
P(m_1+im_2)\\
\frac{1}{2}|P|(h+ik)
\end{pmatrix}\in\Gamma$, its inverse with respect to $\ast$ is
\[
\gamma^{-1}=\begin{pmatrix} -\lambda a-i2\pi b\\
-\begin{pmatrix}
e^{\lambda a +i2\pi b} & 0\\
0 & e^{-\lambda a -i2\pi b}
\end{pmatrix}P(m_1+im_2)\\
-\frac{1}{2}|P|(h+ik)
\end{pmatrix}=\begin{pmatrix} -\lambda a-i2\pi b\\
-\Lambda^{-a}P(m_1+im_2)\\
-\frac{1}{2}|P|(h+ik)
\end{pmatrix}\begin{pmatrix} -\lambda a-i2\pi b\\
-PA^{-a}(m_1+im_2)\\
-\frac{1}{2}|P|(h+ik)
\end{pmatrix}\in\Gamma.
\]
Therefore $\Gamma_{\pi}$ is a discrete subgroup of $G$ and it is straightforward to prove that $M_{\pi}:=\Gamma_{\pi}\backslash G$ is a compact manifold, that is, $\Gamma_{\pi}$ is a lattice of $G$.

The proof that $\Gamma_{\frac{\pi}{2}}$ is a lattice of $G$ is similar and it is omitted.
\end{proof}
As a consequence, $M_{\mu}:=\Gamma_\mu\backslash G$ is a compact complex parallelizable manifold, for $\mu=\pi$ and $\mu=\frac{\pi}{2}$.
\begin{rem}
The map $\pi\colon G\rightarrow \C$ defined that $\pi(z_1, \dots, z_4)=z_1$ induces a holomorphic map $$\pi\colon M_{\mu}\rightarrow (\lambda\Z+i\mu \Z)\backslash\C$$ with fiber biholomorphic to the Iwasawa manifold $\left(H(3;\C)\cap GL(3;\Z[i])\right)\backslash H(3;\C)$.
\end{rem}

We now make us of the results in \cite{AngKas} to construct the subcomplex $C_{\Gamma_\mu}^{\bullet.\bullet}$ of $\mathcal{A}^{\bullet,\bullet}(M_{\mu})$ which computes the Bott-Chern cohomology of $M_{\mu}$. We recall that for every $g\in G$, the adjoint respresentation of $G$
\[
\text{Ad}\colon G\rightarrow \text{Aut}(\mathfrak{g})
\]
 is defined as $\text{Ad}_g:=d(L_{g}\circ R_{g^{-1}})_e$, where $e=(0,0,0,0)$ is the identity element of $G$. By restricting $\text{Ad}$ to $\C$, for every $y_1\in \C$ we have that $(L_{y_1}\circ R_{y_1^{-1}})(z_1,z_2,z_3,z_4)=(z_1,e^{-y_1}z_2,e^{y_1}z_3,z_4)$, so that $$\text{Ad}_{y_1}=d(\text{Id}_{\C},\phi(y_1)).$$ Hence, for any $y_1\in \C$, we obtain the following expression of the semisimple part of $\text{Ad}_{y_1}$ with respect to the frame $\{Z_1,Z_2,Z_3,Z_4\}$
\[
(\text{Ad}_{y_1})_s=\text{diag}(1,e^{-y_1},e^{y_1},1),
\]
where $\alpha_1\equiv \alpha_4\equiv 1$, $\alpha_3(y_1)=e^{-y_1}$, and $\alpha_4(y_1)=e^{y_1}$ are characters of $\C$. Let us denote by the same symbols $\alpha_1,\alpha_2,\alpha_3,\alpha_4$ their trivial extentions to characters of $G$.

Let us then define, as in \cite[Theorem 2.20]{AngKas}, the subcomplex of $\mathcal{A}^{0,\bullet}(M_\mu)$ given by 
\[
B_{\Gamma_{\mu}}^{\bullet}:=\left\langle\left(\frac{\overline{\alpha}_I}{\alpha_I}\right)\varphi^{\overline{I}}: I=\{i_1,\dots, i_k\}\subset \{1,\dots 4\}\,\, \text{such that}\,\, \left(\frac{\overline{\alpha}_I}{\alpha_I}\right)\restrict{\Gamma_{\mu}}=1 \right\rangle,
\]
where $\alpha_I:=\alpha_{i_1}\cdots \alpha_{i_k}$.

\subsubsection{Case: $\mu=\pi$.}
Note that $\frac{\overline{\alpha}_2}{\alpha_2}\restrict{\Gamma_\pi}=e^{z_1-\overline{z_1}}\restrict{\Gamma_\pi}=e^{2\pi ib}=1$ and, analogously, $\frac{\overline{\alpha}_2}{\alpha_2}\restrict{\Gamma_\pi}=e^{\overline{z}_1-z_1}\restrict{\Gamma_\pi}=e^{-2\pi ib}=1$.  Hence, we can describe explicity $B_{\Gamma_\pi}^{\bullet}$ as
\begin{align*}
&B_{\Gamma_\pi}^0=\langle1\rangle\\
&B_{\Gamma_\pi}^{1}=\langle\varphi^{\overline{1}},e^{z_1-\overline{z}_1}\varphi^{\overline{2}},e^{\overline{z}_1-z_1}\varphi^{\overline{3}},\varphi^{\overline{4}}\rangle\\
&B_{\Gamma_\pi}^{2}=\langle e^{z_1-\overline{z}_1}\varphi^{\overline{12}},e^{\overline{z}_1-z_1}\varphi^{\overline{13}},\varphi^{\overline{14}},\varphi^{\overline{23}},e^{z_1-\overline{z}_1}\varphi^{\overline{24}},e^{\overline{z}_1-z_1}\varphi^{\overline{34}}\rangle\\
&B_{\Gamma_\pi}^{3}=\langle\varphi^{\overline{123}},e^{z_1-\overline{z}_1}\varphi^{\overline{124}},e^{\overline{z}_1-z_1}\varphi^{\overline{134}},\varphi^{\overline{234}}\rangle\\
&B_{\Gamma_\pi}^{4}=\langle\varphi^{1234}\rangle.
\end{align*}
\subsubsection{Case: $\mu=\frac{\pi}{2}$.}
In this case, $\frac{\overline{\alpha}_2}{\alpha_2}\restrict{\Gamma_{\frac{\pi}{2}}}=e^{z_1-\overline{z_1}}\restrict{\Gamma_{\frac{\pi}{2}}}=e^{i\pi b}\not\equiv 1$ and, analogously, $\frac{\overline{\alpha}_2}{\alpha_2}\restrict{\Gamma_\pi}=e^{-\pi ib}\not\equiv 1$.  In particular, unlike the previous case, the function $e^{z_1-\overline{z}_1}$ and its inverse are not defined on $M_{\frac{\pi}{2}}$. Hence, we can describe explicity $B_{\Gamma_{\frac{\pi}{2}}}^{\bullet}$ as
\begin{align*}
&B_{\Gamma_{\frac{\pi}{2}}}^0=\langle1\rangle\\
&B_{\Gamma_{\frac{\pi}{2}}}^{1}=\langle\varphi^{\overline{1}},\varphi^{\overline{4}}\rangle\\
&B_{\Gamma_{\frac{\pi}{2}}}^{2}=\langle \varphi^{\overline{14}},\varphi^{\overline{23}}\rangle\\
&B_{\Gamma_{\frac{\pi}{2}}}^{3}=\langle\varphi^{\overline{123}},\varphi^{\overline{234}}\rangle\\
&B_{\Gamma_{\frac{\pi}{2}}}^{4}=\langle\varphi^{\overline{1234}}\rangle.
\end{align*}

\noindent Furthermore, let us define the subcomplex of $\mathcal{A}^{\bullet,\bullet}(M_\mu)$ given by
\[
C_{\Gamma_{\mu}}^{\bullet_1,\bullet_2}:=\textstyle\bigwedge^{\bullet_1}\mathfrak{g}_+^*\otimes_\C B_{\Gamma_{\mu}}^{\bullet_2}+\overline{B}_{\Gamma_{\mu}}^{\bullet_1}\otimes_\C\bigwedge^{\bullet_2}\mathfrak{g}_-^*.
\]
From \cite[Theorem 2.22]{AngKas}, the inclusion
\[
C_{\Gamma_{\mu}}^{\bullet,\bullet}\hookrightarrow \mathcal{A}^{\bullet,\bullet}(M_\mu)
\]
induces the isomorphism
\[
H(C_{\Gamma_{\mu}}^{\bullet-1,\bullet-1}\xrightarrow{\del\delbar}C_{\Gamma_{\mu}}^{\bullet,\bullet}\xrightarrow{\del+\delbar}C_{\Gamma_{\mu}}^{\bullet+1,\bullet}\oplus C_{\Gamma_{\mu}}^{\bullet,\bullet+1})\xrightarrow{\cong} H_{BC}^{\bullet,\bullet}(M_\mu).
\]
As a result, we are able to describe the Bott-Chern cohomology of $M_{\mu}$ for $\mu=\pi$ and $\mu=\frac{\pi}{2}$, see Table \ref{table:1} and Table \ref{table:2}.

\begin{theorem}
The solvmanifold $M_{\mu}$ admits a non vanishing quadruple $ABC$-Massey product for $\mu=\pi$ and $\mu=\frac{\pi}{2}$.
\end{theorem}
\begin{proof}
We will exhibit explictly the non vanishing quadruple $ABC$-Massey product on $M_{\pi}$ and $M_{\frac{\pi}{2}}$. The construction will not depend on the choice of the lattice $\Gamma_{\pi}$ or $\Gamma_{\frac{\pi}{2}}$, so that we will use the generic notation $M_\mu$. Let us consider the following invariant forms on $M_\mu$
\[
\alpha=\varphi^{12}, \quad \beta=\varphi^{\overline{23}}, \quad  \gamma=\varphi^{\overline{13}}, \quad \delta=\varphi^{\overline{12}}.
\]
By structure equations \eqref{eq:struct_eq_quad}, they are all $d$-closed forms and with respect to any Hermitian metric $g$ on $\Gamma\backslash G$, they are all $\del\delbar\ast_g$-closed by bidegree reasons. Hence, the forms $\alpha, \beta,\gamma$, and $\delta$ are all Bott-Chern harmonic and define non vanishing Bott-Chern cohomology classes
$$
[\alpha]\in H_{BC}^{2,0}(M_\mu),\quad [\beta],[\gamma],[\delta]\in H_{BC}^{0,2}(M_\mu).
$$
From now on, we will fix $g$ to be the invariant diagonal Hermitian metric on $(M_\mu,J)$ with fundamental form $\omega=\frac{i}{2}\sum_{j=1}^4\varphi^{j\overline{j}}$.

Let us then consider the quadruple $ABC$-Massey product on $M_{\mu}$
\[
\mathcal{P}:=\Bigl\langle[\varphi^{12}],[\varphi^{\overline{23}}],[\varphi^{\overline{13}}],[\varphi^{\overline{12}}]\Bigr\rangle_{ABC}.
\]
We claim that $\mathcal{P}$ is a well defined non vanishing quadruple $ABC$-Massey product.

In fact, since
\[
\alpha\wedge\beta=\del\delbar\varphi^{2\overline{4}}, \quad \beta\wedge\gamma=0, \quad \gamma\wedge\delta=0,
\]
we can set $x:=\varphi^{2\overline{4}}$, and $y:=0$, $z:=0$.
Note that
$$
x\wedge\gamma-\alpha\wedge y=x\wedge \gamma=\varphi^{2\overline{134}}, \quad y\delta-\beta z=0,
$$
hence, we can fix
\[
\eta=0, \quad \eta'=\varphi^{2\overline{34}}, \quad \xi=\xi'=0.
\]
Thus, by the formula \eqref{eq:ABC_quad}, $\mathcal{P}$ is we defined and represented by the 
\[
\varphi^{2\overline{1234}}.
\]
Moreover, with respect to $g$, the form $\ast\varphi^{2\overline{1234}}=\varphi^{134}$ is $d$-closed, i.e., the form $\varphi^{2\overline{1234}}$ is $\Delta_A$-harmonic and its Aeppli cohomology class $[\varphi^{2\overline{1234}}]_A$ is non vanishing. Furthermore, the space $H_{S^{-1}_{p,q}}$ defined by \eqref{eq:H_S_-1_pq_quad} restricts in our case as
\[
H_{S_{2,6}}^{-1}(\mathcal{A}^{\bullet
,\bullet}(M_{\mu}))=\frac{\ker(\text{pr}\circ d\colon \mathcal{A}^{1,4}(M_\mu)\rightarrow \{ 0\})}{\im(\text{pr}\circ d\colon \mathcal{A}^{0,4}(M_\mu)\oplus \mathcal{A}^{1,3}(M_\mu)\rightarrow \mathcal{A}^{1,4}(M_\mu))},
\]
and hence coincides with the Aeppli cohomology space $$H_A^{1,4}(M_\mu)=\frac{\ker(\del\delbar\mathcal{A}^{1,4}(M_\mu)\rightarrow \{0\})}{\im\del\restrict{\mathcal{A}^{0,4}(M_\mu)}
+\im\delbar\restrict{\mathcal{A}^{1,3}(M_\mu)}}.$$ As a consequence, the representative $\varphi^{1\overline{1234}}$ defines a non vanishing element in $H_{S_{2,6}}^{-1}(\mathcal{A}^{\bullet,\bullet}(M_\mu))$.
Furthermore, since $\varphi^{1\overline{1234}}\in\mathcal{H}_{\Delta_A}^{1,4}(M_\mu)$, no differential from $S_{2,6}^{-2}:=\mathcal{A}^{0,4}(M_\mu)$ can be a $\del$-primitive of $\varphi^{1\overline{1234}}$, yielding that $[\varphi^{1\overline{1234}}]$ is not trivial also as a equivalence class of $H_{S_{2,6}}^{-1}(\mathcal{A}^{\bullet,\bullet}(M_\mu))$. As a result, the product $\mathcal{P}$ is a well defined non vanishing quadruple $ABC$-Massey product on $M_\mu$. 
\end{proof}

\begin{align}\label{table:2}
\begin{array}{ll}
& \qquad \qquad \qquad \qquad \qquad \qquad \qquad \text{Table \ref{table:2}}\\
\toprule
(p,q) & H_{BC}^{p,q}(M_{\frac{\pi}{2}}) \\
\bottomrule
(0,0) & \C\,\langle\,1\,\rangle \\
\midrule
(1,0) &  \C\,\langle\, dz_1\, \rangle \\
(0,1) & \C\,\langle\, dz_{\overline{1}}\,\rangle \\
\midrule
(2,0) & \C\,\langle\, e^{z_1}dz_{12}, \,e^{-z_1}dz_{13},\,dz_{23}\,\rangle\\
(1,1) & \C\,\langle\, dz_{1\overline{1}}\,\rangle\\
(0,2) & \C\,\langle\, e^{\overline{z}_1}dz_{\overline{12}}, \,e^{-\overline{z}_1}dz_{\overline{13}},\,dz_{\overline{23}}\,\rangle\\
\midrule
(3,0) & \C\,\langle\, dz_{123},\, e^{z_1}dz_{12}\wedge \varphi^4, e^{-z_1}dz_{13}\wedge\varphi^4, dz_{23}\wedge\varphi^4\,\rangle\\
(2,1) & \C\,\langle\, e^{z_1}dz_{12\overline{1}},\, e^{-z_1}dz_{13\overline{1}},\, dz_{23\overline{1}}\, \rangle\\
(1,2) & \C\,\langle\,e^{\overline{z}_1}dz_{1\overline{12}},\, e^{-\overline{z}_1}dz_{1\overline{13}},\, dz_{1\overline{23}},\,\rangle\\
(0,3) & \C\,\langle\, dz_{\overline{123}},\, e^{\overline{z}_1}dz_{\overline{12}}\wedge\varphi^{\overline{4}}, e^{-\overline{z}_1}dz_{\overline{13}}\wedge\varphi^{\overline{4}}, \, dz_{\overline{23}}\wedge\varphi^{\overline{4}}\,\rangle\\
\midrule
(4,0) & \C\,\langle\, dz_{123}\wedge\varphi^4\,\rangle\\
(3,1) & \C\,\langle dz_{123\overline{1}},\,  e^{z_1}dz_{12}\wedge\varphi^4\wedge dz_{\overline{1}}, \, e^{-z_1}dz_{13}\wedge\varphi^4\wedge dz_{\overline{1}}, \, dz_{23}\wedge\varphi^4\wedge dz_{\overline{1}}\,\rangle\\
(2,2) & \,0\\
(1,3) & \C\,\langle\, dz_{1\overline{123}},\, e^{\overline{z}_1}dz_{1\overline{12}}\wedge\varphi^{\overline{4}},\, e^{-\overline{z}_1}dz_{1\overline{13}}\wedge\varphi^{\overline{4}},\,dz_{1\overline{23}}\wedge\varphi^{\overline{4}}\,\rangle\\
(0,4) & \C\,\langle\,\, dz_{\overline{123}}\wedge\varphi^{\overline{4}}\,\rangle\\
\midrule
(4,1) & \C\,\langle\, dz_{123}\wedge\varphi^4\wedge dz_{\overline{1}}\,\rangle\\
(3,2) & \C\,\langle\, e^{\overline{z}_1}dz_{23}\wedge\varphi^4\wedge dz_{\overline{12}}, \, e^{-\overline{z}_1}dz_{23}\wedge\varphi^4\wedge dz_{\overline{13}}, \, dz_{23}\wedge\varphi^4\wedge dz_{\overline{23}} \,\rangle\\
(2,3) & \C\,\langle\, e^{z_1}dz_{12\overline{23}}\wedge\varphi^{\overline{4}}, \, e^{-z_1}dz_{13\overline{23}}\wedge\varphi^{\overline{4}}, \, dz_{23}\wedge\varphi^4\wedge dz_{\overline{23}} \,\rangle\\
(1,4) & \C\,\langle\, dz_{1\overline{123}}\wedge\varphi^{\overline{4}}\,\rangle\\
\midrule
(4,2) & \C\,\langle\, e^{\overline{z}_1}dz_{123}\wedge\varphi^4\wedge dz_{\overline{12}}, \, e^{-\overline{z}_1}dz_{123}\wedge\varphi^4\wedge dz_{\overline{13}}, \, dz_{123\wedge}\varphi^4\wedge dz_{\overline{123}}\,\rangle \\
(3,3) & \C\,\langle\, dz_{123\overline{23}}\wedge\varphi^{\overline{4}}, \, e^{z_1}dz_{12}\wedge\varphi^4\wedge dz_{\overline{23}}\wedge\varphi^{\overline{4}}, \, e^{-z_1}dz_{13}\wedge \varphi^4\wedge dz_{\overline{23}}\wedge \varphi^{\overline{4}},\\
& \quad \,\, \,e^{\overline{z}_1}dz_{23}\wedge\varphi^4\wedge dz_{\overline{12}}\wedge \varphi^{\overline{4}},\, e^{-\overline{z}_1}dz_{23}\wedge\varphi^4\wedge dz_{13}\wedge\varphi^{\overline{4}}, \, dz_{23}\wedge\varphi^4\wedge dz_{\overline{23}}\wedge \varphi^{\overline{4}} \,\rangle\\
(2,4) & \C\,\langle\,e^{z_1}dz_{12\overline{123}}\wedge\varphi^{\overline{4}}, \, e^{-z_1}dz_{13\overline{123}}\wedge\varphi^{\overline{4}},\, dz_{23\overline{123}}\wedge\varphi^{\overline{4}}\,\rangle \\
\midrule
(4,3) & \C\, \langle\,dz_{123}\wedge\varphi^4\wedge dz_{\overline{123}}, \, e^{\overline{z}_1}dz_{123}\wedge\varphi^4\wedge dz_{\overline{12}}\wedge\varphi^{\overline{4}}, \,e^{-\overline{z}_1}dz_{123}\wedge\varphi^4\wedge dz_{\overline{13}}\wedge\varphi^{\overline{4}},\\
& \quad \,\,\, dz_{23\overline{123}}\wedge\varphi^{\overline{4}}\,\rangle \\
(3,4) & \C\,\langle\,dz_{123\overline{123}}\wedge \varphi^{\overline{4}},\, e^{z_1}dz_{12}\wedge\varphi^4\wedge dz_{\overline{123}}\wedge\varphi^{\overline{4}}, \, e^{-z_1}dz_{13}\wedge\varphi^4\wedge dz_{\overline{123}}\wedge\varphi^{\overline{4}},\\
 & \quad \,\,\,  dz_{23}\wedge\varphi^4\wedge dz_{\overline{123}}\wedge\varphi^{\overline{4}}\,\rangle \\
\midrule
(4,4) & \C\,\langle\, dz_{123}\wedge\varphi^4\wedge dz_{\overline{123}}\wedge\varphi^{\overline{4}}  \, \rangle  \\
\bottomrule
\end{array}
\end{align}

\begin{equation}\label{table:1}
\resizebox{.80\textwidth}{!}{$
\begin{array}{ll}
& \qquad \qquad \qquad \qquad \qquad \qquad \qquad \qquad \qquad \text{Table \ref{table:1}}\\
\toprule
(p,q) & H_{BC}^{p,q}(M_{\pi}) \\
\bottomrule \vspace{-0.3cm}\\
(0,0) & \C\langle1\rangle \\
\midrule
(1,0) &  \C\,\langle\, dz_1\, \rangle \vspace{0.2cm} \\
(0,1) & \C\,\langle\, dz_{\overline{1}}\,\rangle  \\
\midrule
(2,0) & \C\,\langle\, e^{z_1}dz_{12},\,e^{-z_1}dz_{13},\, dz_{23}\,\rangle \vspace{0.2cm}\\
(1,1) & \C\,\langle\, dz_{1\overline{1}},\, e^{z_1}dz_{1\overline{2}},\, e^{-z_1}dz_{1\overline{3}},\,e^{\overline{z}_1}dz_{2\overline{1}},\,dz_{2\overline{3}},\,e^{-\overline{z}_1}dz_{3\overline{1}},\,dz_{3\overline{2}}\,\rangle \vspace{0.2cm}\\
(0,2) & \C\,\langle\, e^{\overline{z}_1}dz_{\overline{12}},\,e^{-\overline{z}_1}dz_{\overline{13}},\, dz_{\overline{23}}\,\rangle\vspace{0.1cm}\\
\midrule
(3,0) & \C\,\langle\, dz_{123},\, e^{z_1}dz_{12}\wedge\varphi^4,\, e^{-z_1}dz_{13}\wedge\varphi^4,dz_{23}\wedge\varphi^4\,\rangle \vspace{0.2cm}\\
(2,1) & \C\,\langle\, e^{z_1}dz_{12\overline{1}},\, e^{\overline{z}_1}dz_{12\overline{1}}, \, e^{2z_1}dz_{12\overline{2}},\, dz_{12\overline{3}},\, e^{-z_1}dz_{13\overline{1}},\, e^{-\overline{z}_1}dz_{13\overline{1}},\, dz_{13\overline{2}},\, e^{-2z_1}dz_{13\overline{3}},\,dz_{23\overline{1}}, \vspace{0.2cm}\\
 & \quad\,\,\, e^{\overline{z}_1}dz_2\wedge\varphi^4\wedge dz_{\overline{1}}, \, dz_2\wedge\varphi^4\wedge dz_{\overline{3}},\, e^{-\overline{z}_1}dz_{3}\wedge\varphi^4\wedge dz_{\overline{1}}, \, dz_3\wedge\varphi^4\wedge dz_{\overline{2}}\,\rangle \vspace{0.2cm}\\
(1,2) & \C\,\langle\, e^{z_1}dz_{1\overline{12}},\, e^{\overline{z}_1}dz_{1\overline{12}},\, e^{-z_1} dz_{1\overline{13}},\, e^{-\overline{z}_1}dz_{1\overline{13}},\, dz_{1\overline{23}}, \, e^{z_1}dz_{1\overline{2}}\wedge\varphi^{\overline{4}}, \, e^{-z_1}dz_{1\overline{3}}\wedge\varphi^{\overline{4}},\, e^{2\overline{z}_1}dz_{2\overline{12}}, \vspace{0.2cm}\\
& \quad\,\, \,\,dz_{2\overline{13}},\, dz_{2\overline{3}}\wedge\varphi^{\overline{4}},\, dz_{3\overline{12}}, \, e^{-2\overline{z}_1}dz_{3\overline{13}}, \, dz_{3\overline{2}}\wedge\varphi^{\overline{4}}\, \rangle\vspace{0.2cm}\\
(0,3) & \C\,\langle\, dz_{\overline{123}},\, e^{\overline{z}_1}dz_{\overline{12}}\wedge\varphi^{\overline{4}},\, e^{-\overline{z}_1}dz_{\overline{13}}\wedge\varphi^{\overline{4}},dz_{\overline{23}}\wedge\varphi^{\overline{4}}\, \rangle\vspace{0.1cm}\\
\midrule
(4,0) & \C\,\langle\, dz_{123}\wedge \varphi^{4}\, \vspace{0.2cm}\rangle\\
(3,1) & \C\,\langle\, dz_{123\overline{1}},\, e^{z_1}dz_{123\overline{2}},\, e^{-z_1}dz_{123\overline{3}}, \, e^{z_1}dz_{12}\wedge\varphi^4\wedge dz_{\overline{1}}, \, e^{\overline{z}_1}dz_{12}\wedge\varphi^4\wedge dz_{\overline{1}}, \, e^{2z_1}dz_{12}\wedge\varphi^4\wedge dz_{\overline{2}}, \vspace{0.2cm}\\
& \quad \,\,\, dz_{12}\wedge\varphi^4\wedge dz_{\overline{3}}, \, e^{-z_1}dz_{13}\wedge\varphi^4\wedge dz_{\overline{1}},\, e^{-\overline{z}_1}dz_{13}\wedge\varphi^4\wedge dz_{\overline{1}}, \, dz_{13}\wedge\varphi^4\wedge dz_{\overline{2}}, \vspace{0.2cm}\\
& \quad \,\,\, e^{-2z_1}dz_{13}\wedge\varphi^4\wedge dz_{\overline{3}}, \, dz_{23}\wedge\varphi^4\wedge dz_{\overline{1}}\, \rangle \vspace{0.2cm} \\
(2,2) & \C\,\langle\, e^{2z_1}dz_{12\overline{12}}, e^{2\overline{z}_1}dz_{12\overline{12}}, \, dz_{12\overline{13}}, \, e^{2z_1}dz_{12\overline{2}}\wedge\varphi^{\overline{4}},\, dz_{12\overline{3}}\wedge\varphi^{\overline{4}}, \, dz_{13\overline{12}}, \, e^{-2z_1}dz_{13\overline{13}}, \,  e^{-2\overline{z}_1}dz_{13\overline{13}}\vspace{0.2cm}\\
& \quad \,\,\, dz_{13\overline{2}}\wedge\varphi^{\overline{4}}, \, e^{-2z_1} dz_{13\overline{3}}\wedge\varphi^{\overline{4}}, \, e^{2\overline{z_1}}dz_{2}\wedge\varphi^4\wedge dz_{\overline{12}}, \, dz_2\wedge\varphi^4\wedge dz_{\overline{13}}, \, dz_{2}\wedge\varphi^4\wedge dz_{\overline{3}}\wedge\varphi^{\overline{4}},\vspace{0.2cm}\\
& \quad \,\,\, dz_3\wedge\varphi^4\wedge dz_{\overline{12}},\, e^{-\overline{z}_1}dz_3\wedge \varphi^4\wedge dz_{\overline{2}}\wedge\varphi^{\overline{4}} \rangle \vspace{0.2cm}\\
(1,3) & \C\,\langle\, dz_{1\overline{123}},\, e^{z_1}dz_{1\overline{12}}\wedge\varphi^{\overline{4}},\, e^{\overline{z}_1}dz_{1\overline{12}}\wedge\varphi^{\overline{4}}, \, e^{-z_1}dz_{1\overline{13}}\wedge\varphi^{\overline{4}}, \, e^{-\overline{z}_1}dz_{1\overline{13}}\wedge\varphi^{\overline{4}}, dz_{1\overline{23}}\wedge\varphi^{\overline{4}}, e^{\overline{z}_1}dz_{2\overline{123}},\vspace{0.2cm}\\
&\quad \,\,\,  e^{2\overline{z}_1}dz_{2\overline{12}}\wedge\varphi^{\overline{4}}, \, dz_{2\overline{13}}\wedge\varphi^{\overline{4}}, \, e^{-\overline{z}_1}dz_{3\overline{123}}, \, dz_{3\overline{12}}\wedge\varphi^{\overline{4}}, \, e^{-2\overline{z}_1}dz_{3\overline{13}}\wedge\varphi^{\overline{4}}\,\rangle\vspace{0.2cm}\\
(0,4) & \C\,\langle dz_{\overline{123}}\wedge\varphi^{\overline{4}}\,\rangle\vspace{0.1cm}\\
\midrule
(4,1) & \C\,\langle\, dz_{123}\wedge\varphi^4\wedge dz_{\overline{1}}, \, e^{z_1}dz_{123}\wedge\varphi^4\wedge dz_{\overline{2}}, \, e^{-z_1}dz_{123}\wedge\varphi^4\wedge dz_{\overline{3}}\, \rangle \vspace{0.2cm}\\
(3,2) & \C\,\langle\, e^{z_1}dz_{123\overline{12}},\, e^{-z_1}dz_{123\overline{13}}, \, e^{z_1}dz_{123\overline{2}}\wedge\varphi^{\overline{4}}, \, e^{-z_1}dz_{123\overline{3}}\wedge\varphi^{\overline{4}}, \,  e^{2z_1}dz_{12}\wedge\varphi^4\wedge dz_{\overline{12}},\vspace{0.2cm}\\
& \quad \,\, \, e^{2\overline{z}_1}dz_{12}\wedge\varphi^4\wedge dz_{\overline{12}},\, dz_{12}\wedge\varphi^4\wedge dz_{\overline{13}}, \, e^{2z_1}dz_{12}\wedge\varphi^4\wedge dz_{\overline{2}}\wedge\varphi^{\overline{4}}, \, dz_{12}\wedge\varphi^4\wedge dz_{\overline{3}}\wedge\varphi^{\overline{4}},\vspace{0.2cm}\\
& \quad \,\, \, dz_{13}\wedge\varphi^4\wedge dz_{\overline{12}}, \, e^{-2z_1}dz_{13}\wedge\varphi^4 \wedge dz_{\overline{13}}, \, e^{-2\overline{z}_1}dz_{13}\wedge\varphi^4 \wedge dz_{\overline{13}}, \, dz_{13}\wedge\varphi^4\wedge dz_{\overline{2}}\wedge \varphi^{\overline{4}}, \vspace{0.2cm}\\
& \quad \,\,\, e^{-2z_1}dz_{12}\wedge\varphi^4\wedge dz_{\overline{3}}\wedge \varphi^{\overline{4}}, \, e^{\overline{z}_1} dz_{23}\wedge\varphi^4\wedge dz_{\overline{12}}, \, e^{-\overline{z}_1}dz_{23}\wedge\varphi^4\wedge dz_{\overline{13}},\, dz_{23}\wedge\varphi^4\wedge dz_{\overline{23}}\,\rangle\vspace{0.2cm}\\
(2,3) & \C\, \langle\, e^{\overline{z}_1}dz_{12\overline{123}},\, e^{2z_1}dz_{12\overline{12}}\wedge\varphi^{\overline{4}},\, e^{2\overline{z}_1}dz_{12\overline{12}}\wedge\varphi^{\overline{4}},  \, dz_{12\overline{13}}\wedge\varphi^{\overline{4}},\,e^{z_1}dz_{12\overline{23}}\wedge\varphi^{\overline{4}},\, e^{-\overline{z}_1}dz_{12\overline{123}},\vspace{0.2cm}\\
&\quad \,\,\, dz_{13\overline{12}}\wedge\varphi^{\overline{4}},\, e^{-2z_1}dz_{13\overline{13}}\wedge\varphi^{\overline{4}},\, e^{-2\overline{z}_1}dz_{13\overline{13}}\wedge\varphi^{\overline{4}},\, e^{-z_1}dz_{13\overline{23}}\wedge\varphi^{\overline{4}},\, dz_{23\overline{23}}\wedge\varphi^{\overline{4}},\vspace{0.2cm}\\
& \quad \,\,\, e^{\overline{z}_1}dz_2\wedge\varphi^4\wedge dz_{\overline{123}},\, e^{2\overline{z}_1}dz_2\wedge\varphi^4\wedge dz_{\overline{12}}\wedge\varphi^{\overline{4}},\, dz_2\wedge\varphi^4\wedge dz_{\overline{13}}\wedge\varphi^{\overline{4}},\, e^{-\overline{z}_1}dz_3\wedge\varphi^4\wedge dz_{\overline{123}},\vspace{0.2cm}\\
& \quad \,\,\, dz_3\wedge\varphi^4\wedge dz_{\overline{12}}\wedge\varphi^{\overline{4}}, \, e^{-2\overline{z}_1}dz_3\wedge\varphi^4\wedge dz_{\overline{13}}\wedge\varphi^{\overline{4}}  \,\rangle\vspace{0.2cm}\\
(1,4) & \C\,\langle\, dz_{1\overline{123}}\wedge\varphi^{\overline{4}},\,   e^{\overline{z}_1}dz_{2\overline{123}}\wedge\varphi^{\overline{4}}, \, e^{-\overline{z}_1}dz_{3\overline{123}}\wedge\varphi^{\overline{4}}\,\rangle\vspace{0.1cm}\\
\midrule
(4,2) & \C\,\langle\, e^{z_1}dz_{123}\wedge\varphi^4\wedge dz_{\overline{12}}, \, e^{\overline{z}_1}dz_{123}\wedge\varphi^4\wedge dz_{\overline{12}}, \, e^{-z_1}dz_{123}\wedge\varphi^4\wedge dz_{\overline{13}}, \, e^{-\overline{z}_1}dz_{123}\wedge\varphi^4\wedge dz_{\overline{13}}\,\rangle \vspace{0.2cm}\\
& \quad \,\,\, dz_{123}\wedge\varphi^4\wedge dz_{\overline{23}}, \, e^{z_1}dz_{123}\wedge\varphi^4\wedge dz_{\overline{2}}\wedge\varphi^{\overline{4}}, \, e^{-z_1}dz_{123}\wedge\varphi^4\wedge dz_{\overline{3}}\wedge\varphi^{\overline{4}}\,\rangle\vspace{0.2cm}\\
(3,3) & \C\,\langle\, e^{z_1}dz_{123\overline{12}}\wedge\varphi^{\overline{4}},\, e^{-z_1}dz_{123\overline{13}}\wedge\varphi^{\overline{4}}, \, dz_{123\overline{234}}, \, e^{\overline{z}_1}dz_{12}\wedge\varphi^4\wedge dz_{\overline{123}}, \, e^{2z_1}dz_{12}\wedge\varphi^4\wedge dz_{\overline{12}}\wedge \varphi^{\overline{4}},\vspace{0.2cm}\\
& \quad \,\,\, e^{2\overline{z}_1}dz_{12}\wedge\varphi^4\wedge dz_{\overline{12}}\wedge \varphi^{\overline{4}}, dz_{12}\wedge\varphi^4\wedge dz_{\overline{13}}\wedge\varphi^{\overline{4}}, e^{z_1}dz_{12}\wedge\varphi^4\wedge dz_{\overline{23}}\wedge \varphi^{\overline{4}}, \, e^{-\overline{z}_1}dz_{13}\wedge\varphi^4\wedge dz_{\overline{123}}, \vspace{0.2cm}\\
& \quad \,\,\, dz_{13}\wedge\varphi^4\wedge dz_{\overline{12}}\wedge \varphi^{\overline{4}}, \, e^{-2z_1}dz_{13}\wedge\varphi^4\wedge dz_{\overline{13}}\wedge\varphi^{\overline{4}}, \, e^{-2\overline{z}_1}dz_{13}\wedge\varphi^4\wedge dz_{\overline{13}}\wedge\varphi^{\overline{4}},\vspace{0.2cm}\\
& \quad \,\,\, e^{-z_1}dz_{13}\wedge \varphi^4\wedge dz_{\overline{23}}\wedge \varphi^{\overline{4}}, \, dz_{23}\wedge\varphi^4\wedge dz_{\overline{123}}, \, e^{\overline{z}_1}dz_{23}\wedge\varphi^4\wedge dz_{\overline{12}}\wedge\varphi^{\overline{4}}, \, e^{-\overline{z}_1}dz_{23}\wedge\varphi^4\wedge dz_{\overline{13}}\wedge\varphi^{\overline{4}},\vspace{0.2cm}\\
& \quad \,\,\, dz_{23}\wedge\varphi^4\wedge dz_{\overline{23}}\wedge\varphi^{\overline{4}}\rangle\vspace{0.2cm}\\
(2,4) & \C\,\langle\, e^{z_1}dz_{12\overline{123}}\wedge \varphi^{\overline{4}},\, e^{\overline{z}_1}dz_{12\overline{123}}\wedge \varphi^{\overline{4}}, \, e^{-z_1}dz_{13\overline{123}}\wedge \varphi^{\overline{4}}, \,  \, e^{-\overline{z}_1}dz_{13\overline{123}}\wedge \varphi^{\overline{4}}, \, dz_{23\overline{123}}\wedge\varphi^{\overline{4}},\vspace{0.2cm}\\
& \quad \,\,\, e^{\overline{z}_1}dz_2\wedge\varphi^4\wedge dz_{123}\wedge\varphi^{\overline{4}}, \, e^{-\overline{z}_1}dz_3\wedge\varphi^4\wedge dz_{123}\wedge\varphi^{\overline{4}} \,\rangle\vspace{0.1cm}\\
\midrule
(4,3) & \C\,\langle\, dz_{123}\wedge\varphi^4\wedge dz_{\overline{123}}, \, e^{z_1}dz_{123}\wedge\varphi^4\wedge dz_{\overline{12}}\wedge\varphi^{\overline{4}}, \, e^{\overline{z}_1}dz_{123}\wedge\varphi^4\wedge dz_{\overline{12}}\wedge\varphi^{\overline{4}},\, e^{-z_1}dz_{123}\wedge\varphi^4\wedge dz_{\overline{13}}\wedge\varphi^{\overline{4}}, \vspace{0.2cm}\\
& \quad \,\,\, e^{-\overline{z}_1}dz_{123}\wedge\varphi^4\wedge dz_{\overline{13}}\wedge\varphi^{\overline{4}}, dz_{123}\wedge\varphi^4\wedge dz_{\overline{23}}\wedge\varphi^{\overline{4}}\,  \,\rangle\vspace{0.2cm}\\
(3,4) & \C\,\langle\, dz_{123\overline{123}}\wedge\varphi^{\overline{4}}, \, e^{z_1}dz_{12}\wedge\varphi^4\wedge dz_{123}\wedge\varphi^{\overline{4}}, \, e^{\overline{z}_1}dz_{12}\wedge\varphi^4\wedge dz_{123}\wedge\varphi^{\overline{4}} \,, e^{-z_1}dz_{13}\wedge\varphi^4\wedge dz_{123}\wedge\varphi^{\overline{4}},  \vspace{0.1cm}\\
& \quad \,\, \, e^{-\overline{z}_1}dz_{13}\wedge\varphi^4\wedge dz_{123}\wedge\varphi^{\overline{4}}, \, dz_{23}\wedge\varphi^4\wedge dz_{123}\wedge\varphi^{\overline{4}} \rangle\vspace{0.1cm}\\
\midrule
(4,4) & \C\,\langle\, dz_{123}\wedge\varphi^4\wedge dz_{\overline{123}}\wedge\varphi^{\overline{4}}\,\rangle \\
\bottomrule
\end{array}$}
\end{equation}

\end{document}